\documentclass{article}

\usepackage{amsmath}
\usepackage[boxed, linesnumbered, ruled]{algorithm2e}
\usepackage{amsfonts}
\usepackage{graphicx}
   \usepackage{amsmath, amssymb, amsthm}
\usepackage{mathrsfs}
\newcommand\Fontvix{\fontsize{8}{10}\selectfont}
\usepackage{tikz}
 \newtheorem{thm}{Theorem}[section]
 \newtheorem{cor}[thm]{Corollary}
 \newtheorem{lem}[thm]{Lemma}
 \newtheorem{rmk}[thm]{Remark}
 \newtheorem{prop}[thm]{Proposition}
 \newtheorem{obs}[thm]{Observation}
 \theoremstyle{definition}
 \newtheorem{defn}[thm]{Definition}
 \newtheorem{exmp}[thm]{Example}

 \newtheorem{rem}[thm]{Remark}

\numberwithin{equation}{section}

 \newcommand{\To}{\longrightarrow}
 \newcommand{\Map}[3]{#1\, :\, #2\To #3}

 \newcommand{\Real}{\mathbb{R}}

 \newcommand{\set}[1]{\left\{#1\right\}}
 \newcommand{\Set}[2]{\set{#1\ \vert\ #2}}

\begin{document}

\title{On central-max-point tolerance graphs and some subclasses of interval catch digraphs}
\author{Sanchita Paul\thanks{Department of Mathematics, Jadavpur University, Kolkata - 700 032, India. sanchitajumath@gmail.com} \ \thanks{(Corresponding author)} and
Shamik Ghosh\thanks{Department of Mathematics, Jadavpur University, Kolkata - 700 032, India. ghoshshamik@yahoo.com}}

\date{}


\maketitle

\null
\vspace{-4.5em}

\begin{abstract}
\noindent
{\footnotesize Max-point-tolerance graphs (MPTG) were studied by Catanzaro et al.~in 2017 and the same class of graphs were introduced in the name of $p$-BOX($1$) graphs by Soto and Caro in 2015. In our paper we consider {\em central-max-point tolerance graphs} (central MPTG) by taking the points of MPTG as center points of their corresponding intervals. In course of study on this class of graphs we show that the class of central MPTG is same as the class of unit max-tolerance graphs. We prove the class of {\em unit central max-point tolerance graphs} is same as that of {\em proper central max-point tolerance graphs} and both of them are equivalent to the class of proper interval graphs. Next we introduce {\em $50\%$ max-tolerance graphs} and separate this class from unit max-tolerance graph whereas for min-tolerance graphs $50\%$ and unit denote the same graph class. 


\vspace{0.3em}

\noindent Interval catch digraphs (ICD) was introduced by Maehera in 1984. In his introducing paper Maehera proposed a conjecture for the characterization of {\em central interval catch digraph} (central ICD) in terms of forbidden subdigraphs. In this paper, we disprove the conjecture by showing counter examples. We find close relation between a central MPTG and a central ICD.  Also we characterize this digraph by defining a suitable mapping from the vertex set to the real line.
Next we study {\em oriented interval catch digraphs} (oriented ICD). We characterize augmented adjacency matrix of an oriented ICD when it is a tournament. Also we characterize an oriented ICD whose underlying graph is a tree. Further we obtain characterization of the adjacency matrix of a {\em proper interval catch digraph}. Another important result is to provide forbidden subdigraph characterization of those {\em proper oriented interval catch digraphs} whose underlying undirected graph is chordal. In last section we discuss relationships between these classes of graphs and digraphs.} 
\end{abstract}



\section{Introduction}
{\em Intersection graphs} have many important applications in problems related to real-world situations. Because of its diverse applications in network science, different kinds of intersection graphs were introduced in modeling various geometric objects. Among them interval graph is the most important one. A simple graph $G = (V, E)$ is an \textit{interval graph} if one can map each vertex into an interval on the real line so that any two vertices are adjacent if and only if their corresponding intervals intersect. The class of interval graphs was initially posed by Haj\"{o}s in 1957 \cite{Hajos}.
In 1959, the molecular biological scientist Benzer \cite{Benzer} used the model of interval graphs to obtain a physical map from information on pairwise overlaps of the fragments of DNA.
Since then interval graphs were well studied by many people in Computer Science and Discrete Mathematics for its wide application. Many combinatorial problems have been solved for interval graphs in linear time.

\vspace{0.3em}

\noindent Due to their lot of applications in theories and practical situations the graph class was generalized to several variations. In one direction it went in developing concepts of probe interval graphs \cite{Sg}, circular-arc graphs \cite{BDGS}, interval digraphs \cite{Sen}. On the other hand in 1982, Golumbic and Monma introduced the concept of min-tolerance graphs (commonly known as \textit{tolerance graphs}) \cite{Golumbic1}. We denote the length of an interval $I$ on the real line by $|I|$. A simple undirected graph $G=(V,E)$ is a {\em min-tolerance graph} if each vertex $u\in V$ corresponds to a real interval $I_u$ and a positive real number $t_u$, called {\em tolerance}, such that $uv$ is an edge of $G$ if and only if $|I_u\cap I_v|\geq \min \set{t_u,t_v}$. Golumbic in \cite{Golumbic2} introduced \textit{max-tolerance graphs} where each vertex $u\in V$ corresponds to a real interval $I_u$ and a  positive real number $t_u$ (known as tolerance) such that $uv$ is an edge of $G$ if and only if $|I_u\cap I_v|\geqslant \max\set{t_u,t_v}$. For max-tolerance graphs we may assume $t_{u}\leq |I_{u}|$ for each $u\in V$ otherwise $u$ becomes isolated. An max-tolerance graph is a {\em unit-max-tolerance graph} if $|I_u|=|I_v|$ for all $u,v\in V$. Some combinatorial problems like finding maximal cliques were obtained in polynomial time whereas the recognition problem was proved to be NP-hard \cite{Kaufmann} for max-tolerance graphs in 2006. Also a geometrical connection of max-tolerance graphs to semi-squares was obtained in \cite{Kaufmann}. For further details of tolerance graphs one can see \cite{Golumbic2}.

\vspace{0.3em}
\noindent In 2015 Soto and Caro \cite{Soto} introduced a new graph class, namely {\em $p$-BOX} graphs where each vertex corresponds to a box and a point within it in the $d$-dimensional Euclidean space.  Any two vertices are adjacent if and only if the intersection of their corresponding boxes contains both the corresponding points. When the dimension is one the graph class is denoted by $p$-BOX($1$). In 2017 this dimension one graphs are studied independently by Catanzaro et al \cite{Catanzaro}, but with a different name, {\em max-point tolerance graphs} (MPTG) where each vertex $u\in V$ corresponds to a pair of an interval and a point $(I_u,p_u)$, where $I_u$ is an interval on the real line and $p_u\in I_u$, such that $uv$ is an edge of $G$ if and only if $\set{p_u,p_v}\subseteq I_u\cap I_v$. The graphs MPTG have many practical applications in human genome studies and modelling of telecommunication networks \cite{Catanzaro}. A graph $G=(V,E)$ is called {\em central-max-point-tolerance graph} (central MPTG) if $p_u$ is the center point of $I_u$ for each $u\in V$. This graph class actually matches with the graph defined as \textit{$c$-$p$-BOX(1) graph} in \cite{Soto}. It is known that $c$-$p$-BOX($1$) graphs are max-tolerance graphs. We use the terms MPTG and central MPTG for $p$-BOX($1$) and $c$-$p$-BOX($1$) graphs throughout the paper.

\vspace{0.3em}
\noindent In 1984 \cite{Maehera} Maehera defined an analogous concept of intersection graphs for digraphs. He introduced {\em catch digraph} of $F$ as a digraph $G=(V,E)$ where $uv\in E$ if and only if $u\neq v$ and $p_{v}\in S_{u}$ where $F=\{(S_{u},p_{u})|u\in V\}$ is a family of pointed sets in a Euclidean space. The digraph $G$ is said to be {\em represented} by $F$. Later on {\em interval catch digraphs} (in brief, ICD) was introduced by Erich Prisner in 1989 \cite{Prisner} where $S_{u}$ is represented by interval $I_{u}$. This graph class has many applications in real world situations like networking and telecommunications. He characterized interval catch digraphs in terms of the absence of diasteroidal triple in the digraph. A digraph $G=(V,E)$ is {\em unilaterally connected} if for each pair of distinct vertices $u,v\in V$, there is a directed path from $u$ to $v$ or from $v$ to $u$ (or both). The {\em underlying graph} of a digraph $G=(V,E)$ is an undirected graph $U(G)=(V^\prime ,E^\prime )$ where $V^\prime =V$ and $E^\prime =\Set{uv}{uv\in E\text{ or }vu\in E}$. An undirected graph $G=(V,E)$ is a {\em caterpillar} if it is a tree with path $(s_{1},s_{2},\hdots,s_{k})$, called the spine of $G$ such that every vertex of $G$ has at most distance one from the spine.

\vspace{0.5em}
\noindent In our paper we prove that central MPTG graphs are same as unit max-tolerance graphs. Incidentally this settles a question raised in the book of Golumbic \cite{Golumbic2} that whether interval graphs are unit max tolerance graphs or not. Moreover we show that a {\em unit central MPTG} is same as a {\em proper central MPTG} and also is same as a proper interval graph. Looking at the definition and the fact that central max point tolerance graphs are also max-tolerance graphs, one may think that central max point tolerance graphs are same as max-tolerance graphs where tolerance values are half of their corresponding interval length. In this paper, we introduce {\em $50\%$ max-tolerance} graphs analogous to the  similar concept for min-tolerance graphs. In case of min-tolerance graphs, unit and $50\%$ are defining the same class of graphs \cite{Bogart}. In our paper we show that for max-tolerance graphs they are not same. In fact, the classes central MPTG and $50\%$ max-tolerance graphs are not comparable, although they contain various classes of graphs (for example $C_{n},n\geq 3$, proper intervals graphs) in common. Finally we find a close relation of central MPTG with  {\em central interval catch digraph} (central ICD). Central ICD is an interval catch digraph where the points $p_{u}$ are the center points of the intervals $I_{u}$. This digraph was introduced by Maehera \cite{Maehera} in name of {\em interval digraph}. We characterize this digraph by defining a suitable mapping from the vertex set to the real line. Moreover we disprove a conjecture proposed by Maehera in \cite{Maehera} by discussing some structural property of the {\em augmented adjacency matrix} of a central ICD and creating counter examples.We define {\em oriented interval catch digraph} (oriented ICD) as an interval catch digraph $G=(V,E)$ where each edge of its underlying graph has exactly one direction in $G$. 
We prove that an oriented ICD is acyclic and study various properties of it. We also show augmented adjacency matrix of an oriented ICD takes a special form when it is a tournament. Next we characterize an oriented ICD whose underlying graph is a tree. A {\em proper interval catch digraph} (proper ICD) is an interval catch digraph where no interval contains other properly. We obtain characterization of the augmented adjacency matrix of a proper ICD. We also find an important result which characterize those {\em proper oriented interval catch digraphs} (proper oriented ICD) whose underlying undirected graph is chordal, in terms of forbidden subdigraphs. In Conclusion Section we show the relations between the subclasses of max-tolerance graph and central MPTG and the digraphs discussed in this paper and list major open problems in this area. Henceforth undirected graphs will be called simply graphs.

\section{Preliminaries}

\noindent The following characterizations is known for interval catch digraphs.

\begin{thm} \cite{Prisner} \label{e1}
Let $G=(V,E)$ be a simple directed graph. Then $G$ is an interval catch digraph if and only if there exists an ordering $``<"$ of $V$ such that
\begin{equation}\label{icd1}
\text{for }x<y<z\in V,\ xz\in E\Longrightarrow xy\in E \ \text{ and }\ zx\in E\Longrightarrow zy\in E.
\end{equation}
\end{thm}

\noindent  A matrix whose entries are only zeros and ones is a {\em binary} matrix. A binary matrix is said to satisfy {\em consecutive $1$'s property for rows} if its columns can be permuted in such a way that $1$'s in each row occur consecutively \cite{G}. For a simple undirected graph $G=(V,E)$, a matrix $A^{*}(G)$ is known as the {\em augmented adjacency matrix} of $G$ if we replace all principal diagonal elements of the adjacency matrix of $G$ by $1$ \cite{G}. 
Now one can check from \cite{Maehera} that with respect to the ordering described in Theorem \ref{e1}, $A^{*}(G)$ satisfy consecutive $1$'s property along rows. We call this ordering an {\em ICD ordering} of $V$. This ordering is not unique for an ICD. For a simple digraph $G=(V,E)$, we denote the set of (closed) neighbors $\Set{v\in V}{uv\in E}\cup\set{u}$ of $u$ by $d^+[u]$ for each $u\in V$. It is clear that elements of $d^+[u]$ are consecutive in an ICD ordering for each $u\in V$ if $G$ is an ICD. Moreover, if $\Set{(I_{u},p_{u})}{u\in V}$ is a pointed interval representation of an ICD $G$, then the points can be made distinct by slight adjustment and the increasing ordering of these points is an ICD ordering of the corresponding vertices. From \cite{ Golumbic2} we get to know that $\{x_{1},x_{2},x_{3}\}\subseteq V$ form a {\em diasteroidal triple} if for every permutation $\sigma$ of $\{1,2,3\}$ there is an $x_{\sigma(1)}$-avoiding \footnote{
A $v_{2}-v_{3}$ chain is $v_{1}$ avoiding if no initial endpoint of an arc of the chain precedes $v_{1}$ \cite{Golumbic2}} $x_{\sigma(2)}-x_{\sigma(3)}$ chain in $G$. 

\begin{thm} \cite{Prisner} \label{diate}
A digraph is an interval catch digraph if and only if it does not contain any diasteroidal triple.
\end{thm}



\vspace{0.3em}
\noindent A {\em proper interval graph} $G$ is an interval graph in which there is an interval representation of $G$ such that no interval contains another properly. A {\em unit interval graph} is an interval graph in which there is an interval representation of $G$ such that all intervals have the same length. Let $G=(V,E)$ is a graph and $v\in V$. Then the set $N[v]=\Set{u\in V}{u\text{ is adjacent to } v}\cup\set{v}$ is the {\em closed neighborhood} of $v$ in $G$. The {\em reduced graph} $\tilde{G}$ is obtained from $G$ by merging vertices having same closed neighborhood. $G(n,r)$ is a graph with $n$ vertices $x_{1}, x_{2}, \hdots, x_{n}$ such that $x_{i}$ is adjacent to $x_{j}$ if and only if $0 <|i-j| \leq r$, where $r<n$ is a positive integer. Among many characterizations of proper interval graph we list the following which will serve our purpose.

\begin{thm} \label{proper1} \cite{West, Looges, Shamik} (pg 387, ex $11.17$) 
Let G = (V,E) be an interval graph. Then the following are equivalent:
\begin{enumerate}
\item G is a proper interval graph.
\item G is a unit interval graph.
\item There exist a linear ordering $<$ on $V$ such that for every choice of vertices $u,v,w$\\ 
$u<v<w$ and $uw\in E$ implies $uv,vw\in E$.
\item $\tilde{G}$ is an induced subgraph of $G(n,r)$ for some positive integers $n,r$ with $n > r$.
\end{enumerate}
\end{thm}

\noindent The following characterization of MPTG is known:

\begin{thm} \label{mptg1} \cite{Catanzaro}
Let $G=(V,E)$ be a simple undirected graph. Then $G$ is an MPTG if and only if there is an ordering of vertices of $G$ such that the following condition holds:
\begin{equation}\label{4p1}
\text{For any } x<u<v<y,\ xv,uy\in E \Longrightarrow uv\in E.
\end{equation}
\end{thm}

\begin{defn}
Let $A=(a_{ij})$ and $B=(b_{ij})$ be two $n\times n$ binary matrices. We define $A \wedge B=(c_{ij})$ where $c_{ij}=a_{ij}\wedge b_{ij}$ with the rules: $0\wedge 0=1\wedge 0=0\wedge 1=0$ and $1\wedge 1=1$.
\end{defn}

\noindent The above characterization leads to the following observations.

\begin{obs}\label{obsmptg}
Let $G$ be a simple undirected graph. Then following are equivalent:
\begin{enumerate}
\item $G$ is an MPTG.
\item There is an ordering of vertices of $G$ such that for any $u<v$, $u,v\in V$,
\begin{equation}
uv\not\in E \Longrightarrow uw\not\in E\text{ for all }w>v\text{ or, }wv\not\in E\text{ for all }w<u.
\end{equation}
\item There exists an ordering of vertices such that every $0$ above the principal diagonal of the augmented adjacency matrix $A(G)$ has either all entries right to it are $0$ or, all entries above it are $0$.
\item There exists a binary matrix $M$ with consecutive $1$'s property for rows such that the augmented adjacency matrix $A(G)=M\wedge M^T$.

\item  There exist an interval catch digraph $D$ such that $G=D\cap D^{T}$ where $D^{T}$ is the digraph obtained from $D$ by reversing direction of every arc.
\end{enumerate}
\end{obs}

\begin{proof}
The condition $2$ is equivalent to (\ref{4p1}) in the other way. Condition $3$ is a matrix version of condition $2$. Condition $4,5$ follows from definition of MPTG and ICD respectively.
\end{proof}

\noindent
Let $G=(V,E)$ be a simple undirected graph and $\emptyset\neq X\subseteq V$. Then $G[X]$ denotes the subgraph of $G$ induced by $X$. In Proposition $6.7$ of \cite{Catanzaro} it is proved that if $G$ is an MPTG with non-adjacent vertices $u$ and $v$, then $G[N(u)\cap N(v)]$ is an interval graph. Also in Proposition $7.1$ of \cite{Kaufmann} it is shown that $\overline{C_{n}}$, $n>9$ is not an max-tolerance graph. We show that these graphs are not MPTG as well.

\begin{obs}\label{ex1}
The complement of a cycle of length greater than $9$ is not an MPTG. 
\end{obs}

\begin{proof}
Suppose on contrary $\overline{C_{n}}$, $n>9$ is an MPTG. Now as $\overline{C_{n}}$, $n>9$ contain the graph $G^{'}_{1}$ with vertices $\{1,4,5,n-3,n-2,n\}$ in Figure \ref{figg1g2} (left) as induced subgraph where common neighbors of $n-2$ and $n-3$ form a chordless $4$-cycle and so is not an interval graph. Thus $G^{'}_{1}$ is not an MPTG. Since any subgraph of an MPTG must be an MPTG, hence the result follows.
\end{proof}

\begin{figure}[b]
\begin{center}
\begin{tikzpicture} [scale=0.7]
\draw[-][draw=black,thick] (1,0) -- (4,0);
\draw[-][draw=black,thick] (1,-1) -- (4,-1);
\draw[-][draw=black,thick] (1,0) -- (1,-1);
\draw[-][draw=black,thick] (4,0) -- (4,-1);
\draw[-][draw=black,thick] (2.5,1) -- (1,0);
\draw[-][draw=black,thick] (2.5,1) -- (4,0);
\draw[-][draw=black,thick] (2.5,-2) -- (1,-1);
\draw[-][draw=black,thick] (2.5,-2) -- (4,-1);
\draw[-][draw=black,thick] (2.5,1) -- (1,-1);
\draw[-][draw=black,thick] (2.5,1) -- (4,-1);
\draw[-][draw=black,thick] (2.5,-2) -- (1,0);
\draw[-][draw=black,thick] (2.5,-2) -- (4,0);
\node [left] at (1,0) {\tiny{$1$}};
\node [right] at (4,0) {\tiny{$4$}};
\node [left] at (1,-1) {\tiny{$5$}};
\node [right] at (4,-1) {\tiny{$n$}};
\node [above] at (2.5,1) {\tiny{$n-3$}};
\node [below] at (2.5,-2) {\tiny{$n-2$}};
\end{tikzpicture} \hspace{1in} \begin{tikzpicture} [scale=0.7]
\draw[-][draw=black,thick] (1,0) -- (4,0);
\draw[-][draw=black,thick] (1,-1) -- (4,-1);
\draw[-][draw=black,thick] (1,0) -- (1,-1);
\draw[-][draw=black,thick] (4,0) -- (4,-1);
\draw[-][draw=black,thick] (2.5,1) -- (1,0);
\draw[-][draw=black,thick] (2.5,1) -- (4,0);
\draw[-][draw=black,thick] (2.5,-2) -- (1,-1);
\draw[-][draw=black,thick] (2.5,-2) -- (4,-1);
\draw[-][draw=black,thick] (2.5,1) -- (1,-1);
\draw[-][draw=black,thick] (2.5,1) -- (4,-1);
\draw[-][draw=black,thick] (2.5,-2) -- (1,0);
\draw[-][draw=black,thick] (2.5,-2) -- (4,0);
\node[above] at (2.5,1.4){\tiny{$v_{5}$}};
\node[below] at (2.5,-2.4){\tiny{$v_{6}$}};
\node[left] at (0.6,0) {\tiny{$v_{1}$}};
\node[right] at (4.4,0){\tiny{$v_{2}$}};
\node[left] at (0.6,-1) {\tiny{$v_{4}$}};
\node[right] at (4.4,-1) {\tiny{$v_{3}$}};
\end{tikzpicture}
\caption{The graphs $G^{'}_1$ in Observation \ref{ex1} and $G^{'}_{2}$ in Example \ref{ex2}}\label{figg1g2}
\end{center}
\end{figure}

\noindent In the following example we show that MPTG and max-tolerance graph are not same. 

\begin{exmp} \label{ex2}
The graph $G^{'}_{2}$ in Figure \ref{figg1g2} (right) with vertex set $V=\{v_{i}|1\leq i\leq 6\}$ is not an MPTG as above. But it is a max-tolerance graph with the following interval and tolerance representation.\\
\noindent $I_{1}=[10,46],t_{1}=21, I_{2}=[20,50],t_{2}=18, I_{3}=[18,49.5],t_{3}=28.5, I_{4}=[15,60],t_{4}=31, I_{5}=[21,52],t_{5}=10, I_{6}=[12,50],t_{6}=30$.  
\end{exmp}

\begin{thm} \cite{Soto}\label{sp}
The graph class central MPTG properly contains the class of interval graphs.
\end{thm}

\section{Central max point tolerance graphs}

\noindent We begin with a trivial but important observation which will be used throughout the rest of the paper.

\begin{obs}
Let $\Set{I_u}{u\in V}$ be a collection of intervals, where $I_u=[\ell_u,r_u]$, $h_u=|I_u|=r_u-\ell_u$ and $c_u=\frac{\ell_u+r_u}{2}$. Then $\set{c_u,c_v}\subseteq I_u\cap I_v$ $\Longleftrightarrow$ $|c_v-c_u|\leqslant \dfrac{1}{2} \text{min}\set{h_u,h_v}$ $\Longleftrightarrow$ $\ell_v\leqslant c_u\leqslant c_v\leqslant r_u$ (for $c_u\leqslant c_v$).
\end{obs}

\begin{proof}
We have $c_{u}\in I_{v}=[l_{v},r_{v}]=[c_{v}-\dfrac{h_{v}}{2},c_{v}+\dfrac{h_{v}}{2}]$ $\Longleftrightarrow$ $c_{v}-\dfrac{h_{v}}{2}\leq c_{u}\leq c_{v}+\dfrac{h_{v}}{2}$ $\Longleftrightarrow$ $-\dfrac{h_{v}}{2}\leq c_{u}-c_{v}\leq \dfrac {h_{v}}{2}$
$\Longleftrightarrow$ $|c_{u}-c_{v}|\leq \dfrac{h_{v}}{2}$. Thus $\{c_{u},c_{v}\}\subseteq I_{u}\cap I_{v}$
$\Longleftrightarrow$ $|c_{u}-c_{v}|\leq \dfrac{1}{2} \text{min}\{h_{u},h_{v}\}$. Also it is clear that $c_{u}\in I_{v}=[l_{v},r_{v}]$ $\Longleftrightarrow$ $l_{v}\leq c_{u}\leq r_{v}$.
\end{proof}

\noindent In the above section we observed that the graph classes of max-point-tolerance graphs and max tolerance graphs are not same. Now it is interesting to see that the classes of central MPTG and unit max-tolerance graphs are same.

\begin{thm}\label{umtg}
Let $G$ be a simple undirected graph. Then $G$ is a central MPTG if and only if $G$ is a unit max-tolerance graph.
\end{thm}

\begin{proof}
Let $G=(V,E)$ be a central MPTG with a central MPTG representation $(I_{u},c_{u})$ where $I_{u}=[l_{u},r_{u}]$, $c_{u}$ be the center point of $I_{u}$ for each vertex $u\in V$. Let $h_{u}=r_{u}-l_{u}$ for all $u\in V$. Choose $h_{0}> \text{max} \Set{h_{u}}{u\in V}$. Define $t_{u}=\dfrac{h_{0}-h_{u}}{2}$, $y_{u}=c_{u}+\dfrac{h_{0}}{2}$ and $T_{u}=[c_{u},y_{u}]$. Note that $t_{u}>0$ and $|T_{u}|=y_{u}-c_{u}=\dfrac{h_{0}}{2}$ which is a constant for all $u\in V$.

\noindent Suppose $uv\in E$ and $c_{v}\leq c_{u}$. Then $c_{u}-c_{v}\leq \dfrac{1}{2} \text{min}\{h_{u},h_{v}\}\leq \dfrac{h_{0}}{2}$. So $c_{u}\leq c_{v}+\dfrac{h_{0}}{2}=y_{v}$. This implies $c_{v}\leq c_{u}\leq y_{v}$. So $T_{u}\cap T_{v}=[c_{u},y_{v}]\neq \emptyset$ and $|T_{u}\cap T_{v}|=y_{v}-c_{u}=c_{v}+\dfrac{h_{0}}{2}-c_{u}=\dfrac{h_{0}}{2}-(c_{u}-c_{v})\geq \dfrac{h_{0}-h_{u}}{2},\dfrac{h_{0}-h_{v}}{2}$. So $y_{v}-c_{u}\geq t_{u},t_{v}$, i.e.,
\begin{eqnarray}
\label{cmp1}
|T_{u}\cap T_{v}|\geq \text{max} \{t_{u},t_{v}\}.
\end{eqnarray}

\noindent On the other hand, (\ref{cmp1}) implies $T_{u}\cap T_{v}\neq \emptyset$ and $c_{v}\leq c_{u}\leq y_{v}\leq y_{u}$. So $|T_{u}\cap T_{v}|=y_{v}-c_{u}$. Now $y_{v}-c_{u}\geq \text{max}\{t_{u},t_{v}\}$ implies $\dfrac{h_{0}}{2} -(c_{u}-c_{v})\geq \text{max} \{\dfrac{h_{0}-h_{u}}{2},\dfrac{h_{0}-h_{v}}{2}\}$. Thus $c_{u}-c_{v}\leq \dfrac{1}{2} \text{min} \{h_{u},h_{v}\}$, i.e., $uv\in E$. Therefore $G$ is a unit max-tolerance graph with interval representation $\{T_{u}=[c_{u},y_{u}]|u\in V\}$ and tolerances $\{t_{u}|u\in V\}$ as defined above.

\vspace{0.5em}\noindent Conversely, let $G=(V,E)$ be a unit max-tolerance graph with interval representation \\
$\Set{T_{u}=[l_{u},r_{u}]}{u\in V}$ and tolerances $\Set{t_{u}}{u\in V}$. Let $h=|T_{u}|$ for all $u\in V$. Define $I_{u}=[l_{u}-(h-t_{u}),l_{u}+(h-t_{u})]$. Then $c_{u}$, the center of $I_{u}=l_{u}$ and $h_{u}=|I_{u}|=2(h-t_{u})<2h$. Suppose $uv\in E$. Then $|T_{u}\cap T_{v}|\geq \text{max} \{t_{u},t_{v}\}$. Now for $l_{u}\leq l_{v}$, $|T_{u}\cap T_{v}|=r_{u}-l_{v}$. Then $r_{u}-l_{v}\geq \text{max}\{t_{u},t_{v}\}$, i.e., $h+l_{u}-l_{v}\geq \text{max}\{t_{u},t_{v}\}$. This implies $l_{v}-l_{u}\leq \text{min}\{h-t_{u},h-t_{v}\}$, i.e., $c_{v}-c_{u}\leq \dfrac{1}{2}\text{min} \{h_{u},h_{v}\}$. Finally, the condition that $0<c_{v}-c_{u}\leq \dfrac{1}{2} \text{min} \{h_{u},h_{v}\} \Rightarrow l_{v}-l_{u}<h$ and $l_{u}\leq l_{v}<l_{u}+h=r_{u}$. So $T_{u}\cap T_{v}\neq \emptyset$ and $|T_{u}\cap T_{v}|=r_{u}-l_{v}$. Then $l_{v}-l_{u}=c_{v}-c_{u}\leq \dfrac{1}{2} \text{min}\{h_{u},h_{v}\}$ implies $|T_{u}\cap T_{v}|\geq \text{max} \{t_{u},t_{v}\}$, i.e., $uv\in E$. Thus $\Set{(I_{u},c_{u})}{u\in V}$ is a central MPTG representation of G.
\end{proof}

\begin{rmk}
{\em In (\cite{Golumbic2}, page $215$) Golumbic wrote that ``Every interval graph is a proper max-tolerance graph. It is not yet known if this can be strengthened to unit max-tolerance." The above theorem shows central MPTG and unit max-tolerance graphs denote the same graph class. Hence from Theorem \ref{sp} one can easily conclude that every interval graph is a unit max-tolerance graph. Thus we settle the above query posed in the book of Golumbic.}
\end{rmk}

\noindent
In the sequel we show that the class of max-tolerance graphs properly contains the class of central MPTG. We begin with the following definition which unfolds more insight in the structure of a central MPTG.

\begin{defn}
(\textit{$C$-order})
Let $G=(V,E)$ be a central MPTG with (distinct) center points $\Set{c_{u}}{u\in V}$ of the intervals $\Set{I_{u}}{u\in V}$ in its central MPTG representation $\Set{(I_{u},c_{u})}{u\in V}$. The $C$-order of the set $V$ is the total order induced by the center points. For convenience abusing notation, henceforth we write $u < v$ if and only if $c_{u}<c_{v}$. 
\end{defn}

\noindent In the following we present a necessary condition for central MPTG.

\begin{thm}\label{4pt}
Let $G=(V,E)$ be a central MPTG. Then there is an ordering $\prec^{*}$ of vertices of $G$ such that the following condition holds:
\begin{equation}\label{cmptg1}
\text{For any } x\prec^{*}u\prec^{*}v\prec^{*}y,\ xv,uy\in E \Longrightarrow uv\in E \text{ and } (xu\in E \text{ or } vy\in E \text{ or } xu,vy\in E).
\end{equation}
\end{thm}

\begin{proof}
Let $G=(V,E)$ be a central MPTG with a central MPTG representation $(I_{u},c_{u})$ for each $u\in V$. We arrange vertices according to the increasing order of center points (i.e., in $C$-order) of representing intervals. Suppose in this ordering we have $x<u<v<y$ and $xv,uy\in E$. Then $c_{v},c_{x}\in I_{v}\cap I_{x}$ and $c_{u},c_{y}\in I_{u}\cap I_{y}$. Also we have $c_{x}<c_{u}<c_{v}<c_{y}$. Now $c_{x},c_{v}\in I_{v}\Rightarrow c_{u}\in I_{v}$ and $c_{u},c_{y}\in I_{u} \Rightarrow c_{v}\in I_{u}$. Therefore $uv\in E$. Again $c_{x},c_{v}\in I_{x} \Rightarrow c_{u}\in I_{x}$ and $c_{u},c_{y}\in I_{y}\Rightarrow c_{v}\in I_{y}$. Thus if $xu,vy \notin E$, then $c_{x} \notin I_{u}$ and $c_{y} \notin I_{v}$. But then $c_{u}-c_{x}>c_{y}-c_{u}$ as $c_{x} \notin I_{u}$ but $c_{y}\in I_{u}$, and $c_{y}-c_{v}>c_{v}-c_{x}$ as $c_{y}\notin I_{v}$ but $c_{x}\in I_{v}$. Combining these inequalities we have $c_{v}<\dfrac{c_{x}+c_{y}}{2}<c_{u}$ which is a contradiction. Therefore $xu\in E$ or, $vy\in E$ or, $xu,vy\in E$.
\end{proof}

\noindent
In \cite{Soto}, it is shown that every cycle $C_n$ of length $n\geq 3$ is a central MPTG. 

\begin{defn}
A cycle $C_{n}$ is said to be {\em circularly consecutive $C$-ordered} if starting from a fixed vertex (say $u$) one can order all its vertices in a circularly consecutive way in clockwise (or anticlockwise) direction until $u$ is reached in a $C$-order.
\end{defn}

\noindent
In particular, for $n=4$, one can obtain a central MPTG representation of $C_{4}$ from \cite{Soto} where vertices are circularly consecutive $C$-ordered. We state a stronger version in the following corollary.

\begin{cor}  \label{c6}
Any induced $C_{4}$ in central MPTG must be circularly consecutive $C$-ordered.
\end{cor}  

\begin{proof}
All other possible $C$-orderings of vertices will violate (\ref{cmptg1}). Hence the proof follows.
\end{proof}

\noindent
Similarly one can obtain the following:

\begin{cor}
Any induced $P_{4}$ in central MPTG must have vertex consecutive ending edges i.e; vertices corresponding to ending edges of $P_{4}$ occur consecutively in a $C$-order (up to permutations between them) at least in one end.
\end{cor}

\begin{thm}\label{c61}
 $\overline{C_{6}}$ is a max-tolerance graph but it is not a central MPTG.
 \end{thm}
 
\begin{proof}
Let $\{v_{i}| 1\leq i\leq 6\}$ be the vertices occurred in circularly consecutive way in clockwise (or anticlockwise) order in $C_{6}$. We assign the following intervals and tolerances for all the vertices so that they satisfy max-tolerance representation in $\overline{C_{6}}$.
 $I_{v_{1}}=[0,20],t_{v_{1}}=10,
 I_{v_{2}}=[12,24],t_{v_{2}}=6, I_{v_{3}}=[0,22],t_{v_{3}}=11, I_{v_{4}}=[9.5,19.5],t_{v_{4}}=5,
 I_{v_{5}}=[7.5,30.5],t_{v_{5}}=11.5, I_{v_{6}}=[10.5,21.5],t_{v_{6}}=5.5$. 
\begin{figure}[t]
\begin{center}
\begin{tikzpicture} [scale=0.9]
\draw[-][draw=black,thick] (1.5,0) -- (3.7,0);
\draw[-][draw=black,thick] (0.9,1) -- (2.9,1);
\draw[-][draw=black,thick] (-1.3,2) -- (3.7,2);
\node [left] at (1,0) {$I_{6}$};
\node [left] at (0.6,1) {$I_{4}$};
\node [left] at (-1.5,2) {$I_{1}$};
\draw [fill=black](2.6,0) circle [radius=0.05];
\draw [fill=black](1.9,1) circle [radius=0.05];
\draw [fill=black](1.2,2) circle [radius=0.05];
\node [left] at (1,0) {$I_{6}$};
\node [below] at (1.5,0) {$a_{6}$};
\node [below] at (3.7,0) {$b_{6}$};
\node [below] at (2.6,0) {$c_{6}$};
\node [below] at (0.8,1) {$a_{4}$};
\node [below] at (2.9,1) {$b_{4}$};
\node [below] at (1.9,1) {$c_{4}$};
\node [below] at (-1.3,2) {$a_{1}$};
\node [below] at (1.5,2) {$c_{1}$};
\node [below] at (3.7,2) {$b_{1}$};
\draw [fill=black](3.3,-1) circle [radius=0.05];
\node [below] at (3.05,-1) {$c_{3}$};
\draw[-][draw=black] (1.2,3) -- (1.2,-5);
\draw[-][draw=black] (3.3,3) -- (3.3,-5);
\node [below] at (0.5,-1) {$a_{3}$};
\draw[-][draw=black,thick] (0.5,-1) -- (5.6,-1);
\node [left] at (0.3,-1) {$I_{3}$};
\node [below] at (1.5,-2) {$a_{2}$};
\node [below] at (3,-2) {$b_{2}$};
\draw[-][draw=black,thick] (1.4,-2) -- (3,-2);
\node [left] at (1,-2) {$I_{2}$};

\node [below] at (5.6,-1) {$b_{3}$};

\draw[-][draw=black,thick] (0.5,-3) -- (3.8,-3);
\node [below] at (0.5,-3) {$a_{5}$};
\node [below] at (3.8,-3) {$b_{5}$};

\node [left] at (0.3,-3) {$I_{5}$};
\end{tikzpicture} 
\caption{Relative positions of intervals described in the proof of Theorem \ref{c61}}\label{figg1g3}
\end{center}
\end{figure}

\vspace*{.2 em}

\noindent Now suppose $\overline{C_{6}}=(V,E)$ is a central MPTG with central MPTG representation $(I_{v},c_{v})$ where $I_{v}=[a_{v},b_{v}]$, $c_{v}$ be the center point of $I_{v}$ for each $v\in V$. Let $\{v_{i}|1\leq i \leq 6\}$ be the vertices occurred in circularly consecutive way in clockwise (or anticlockwise) order in $C_{6}$. It is easy to check that the subgraph induced by deleting the vertices $\{v_{2},v_{5}\}$ from $\overline{C_{6}}$ is a $C_{4}$. Now from  Corollary \ref{c6} we can conclude that the vertices in $C_{4}=\{v_{1},v_{4},v_{6},v_{3}\}$ are circularly consecutive $C$-ordered. Without loss of generality we can take $c_{1}<c_{4}<c_{6}<c_{3}$.
As $v_{1}v_{5},v_{3}v_{5}\in E$, $a_{5}\leq c_{1}$ and $c_{3}\leq b_{5}$.
Thus we get $[c_{1},c_{3}]\subseteq [a_{5},b_{5}]$. Hence
$c_{4}\in [c_{1},c_{3}]$ imply $c_{4}\in I_{5}$. Below we will show $c_{5}\in I_{4}$ (see Figure \ref{figg1g3}) which lead us to contradiction as $v_{4},v_{5}$ are nonadjacent in $\overline{C_{6}}$. 

\vspace*{.2 em}
\noindent 
 As $v_{3}v_{6}\in E$, $c_{3}\leq b_{6}$. Again $v_{1}v_{3}\in E$ imply 
$[c_{1},c_{3}]\subseteq I_{3}$ as $c_{1}<c_{3}$. Moreover $c_{4} \in [c_{1},c_{3}]\subseteq I_{3}$. Hence $c_{3}\notin I_{4}$ as $v_{3}v_{4}\notin E$,  which imply $b_{4}<c_{3}$ as $c_{4}<c_{3}$. Combining we get 
\begin{eqnarray}\label{1}
b_{4}<c_{3}\leq b_{6} 
\end{eqnarray}

\vspace*{.1 em}
\noindent 
Since $v_{1}v_{5}, v_{3}v_{5}\in E$, we get $[c_{1},c_{3}]\subseteq I_{5}$ as $c_{1}<c_{3}$. Now as $c_{6}\in[c_{1},c_{3}]$, $c_{6}\in I_{5}$. Hence  $c_{5}\notin I_{6}$ as $v_{5}v_{6}\notin E$. Hence $c_{5}<a_{6}$ or $c_{5}>b_{6}$. Again $c_{5}\leq b_{4}$ (as $v_{4}v_{5}\in E)$ and hence $c_{5}$ can not be greater than $b_{6}$ as $b_{4}<b_{6}$ from (\ref{1}). Thus we get 
\begin{eqnarray}\label{2}
c_{5}<a_{6}
\end{eqnarray}

\noindent 
Note that $v_{1}v_{3}\in E$ imply $[c_{1},c_{3}]\subseteq I_{1}$ as $c_{1}<c_{3}$. Again $c_{6}\in [c_{1},c_{3}]\subseteq I_{1}$. But as $v_{1}v_{6}\notin E$, $c_{1}<a_{6}$ as $c_{1}<c_{6}$. Again $v_{1}v_{4}\in E$ imply $a_{4}\leq c_{1}$.  Hence combining we get $a_{4}\leq c_{1}<a_{6}$. Again $v_{4}v_{6}\in E$ imply $a_{6}\leq c_{4}$ and $c_{6}\leq b_{4}$. Hence $a_{6}\leq c_{4}<c_{6}\leq b_{4}$. Thus combining these inequalities and using (\ref{1}) one can conclude
 $a_{4}\leq c_{1}<a_{6}\leq c_{4}<c_{6}\leq b_{4}<c_{3}\leq b_{6}$. Thus we get
 \begin{eqnarray}\label{3}
 a_{4}\leq c_{1}<a_{6}<b_{4}<c_{3}\leq b_{6}
\end{eqnarray}																				 
As $v_{2} v_{4}, v_{2}v_{6}\in E$ imply $c_{2}\in I_{4}\cap I_{6}=[a_{6},b_{4}]\subseteq [c_{1},c_{3}]\subseteq I_{1}, I_{3}$ (from (\ref{3})) which imply $c_{2}\in I_{1}, I_{3}$. Now $v_{1}v_{2}\notin E$ imply $a_{2}>c_{1}$ as $c_{1}<c_{2}$ from above. Again $v_{2}v_{3}\notin E$ imply $b_{2}<c_{3}$ as $c_{2}<c_{3}$ from above. Thus we get $c_{1}<a_{2}\leq c_{2}\leq b_{2}<c_{3}$, i.e; 
 \begin{eqnarray}\label{4}
 [a_{2},b_{2}]\subseteq [c_{1},c_{3}].
\end{eqnarray}
\vspace{.2em}\noindent We now show $c_{5}\in I_{4}$. As $v_{2}v_{5}\in E$, $c_{5}\in [a_{2},b_{2}]\subseteq [c_{1},c_{3}]\subseteq [a_{4},b_{6}]$ from (\ref{3}) and (\ref{4}). 
Now using (\ref{2}) one can conclude that $c_{5}$ must belong to $I_{4}$. 

\vspace*{.2 em}
\noindent From above we can conclude now that no CMPTG representation of $\overline{C_{6}}$ can be found with respect to the above $C$-ordering. For other possible $C$-orderings following similar type argument one can reach to contradiction.
\end{proof}

\noindent
Now we present a sufficient condition for an MPTG to be a central MPTG.

\begin{thm}
Let $G=(V,E)$ be an MPTG with $n$ vertices. Let the ordering $\set{v_1,v_2,\ldots,v_n}$ of vertices of $G$ that satisfies (\ref{4p1}) and each $v_i$ corresponds to a natural number $x_i$ such that $x_1<x_2<\cdots <x_n$ and the following conditions hold for all $i=1,2,\ldots,n$:
\begin{eqnarray}
x_{i_2+1}-x_i>x_i-x_{i_1} & \text{when} & i_2<n \label{cmptg2}\\
x_i-x_{i_1-1}>x_{i_2}-x_i & \text{when} & i_1>1 \label{cmptg3}
\end{eqnarray}
where $i_1$ and $i_2$ be the least and the highest indices such that $i_1=i$ or, $v_iv_{i_1}\in E$ and $i_2=i$ or, $v_iv_{i_2}\in E$. Then $G$ is a central MPTG.
\end{thm}

\begin{proof}
Suppose the conditions hold. Define $r_{i}=\text{max}\{x_{i}-x_{i_{1}},x_{i_{2}}-x_{i}\}$ and $I_{i}=[x_{i}-r_{i},x_{i}+r_{i}]$ for $i=1,2,\hdots,n$. We show that $G=(V,E)$ is a central MPTG with an interval representation $\{I_{v_{i}}|i=1,2,\hdots,n\}$ where $V=\{v_{1},v_{2},\hdots,v_{n}\}$ and this ordering of vertices satisfies (\ref{4p1}). Suppose $v_{i}v_{j}\in E$. Then by definition of $i_{1}$ and $i_{2}$, we have $x_{i_{1}}\leq x_{j}\leq x_{i_{2}}$ and $x_{j_{1}}\leq x_{i}\leq x_{j_{2}}$. Then $x_{i}-x_{i_{1}}\geq x_{i}-x_{j}$ and $x_{i_{2}}-x_{i}\geq x_{j}-x_{i}$ which imply $|x_{i}-x_{j}|\leq r_{i}$ and so $x_{j}\in I_{v_{i}}$. Similarly $x_{i}\in I_{v_{j}}$. Hence $\{x_{i},x_{j}\}\subseteq I_{v_{i}}\cap I_{v_{j}}$. Now let $v_{i}v_{j}\notin E$. Without loss of generality we assume $i<j$. Suppose $j_{1}<i$ and $j<i_{2}$. Then we have $j_{1}<i<j<i_{2}$ and $v_{j_{1}}v_{j},v_{i}v_{i_{2}}\in E$. Then by (\ref{4p1}), $v_{i}v_{j}\in E$, which is a contradiction. Thus either $i<j_{1}$ or $j>i_{2}$. Then $i\leq j_{1}-1$ or $j\geq i_{2}+1$. For the first inequality by (\ref{cmptg2}), we have $x_{j}-x_{i}\geq x_{j}-x_{j_{1}-1}> x_{j_{2}}-x_{j}$. Also $x_{j}-x_{i}>x_{j}-x_{j_{1}}$, as $x_{i}<x_{j_{1}}$. Thus $x_{j}-x_{i}>r_{j}$ which implies $x_{i}\notin I_{v_{j}}$. Similarly $j\geq i_{2}+1$ implies $x_{j}\notin I_{v_{i}}$. Therefore $G$ is a central MPTG.
\end{proof}

\begin{defn}
A central MPTG $G=(V,E)$ is called {\em proper} if it has an interval representation with the required condition such that no interval contains another properly. We call it {\em proper-central-max-point tolerance graph} (in brief, proper central MPTG). Similarly, a central MPTG $G=(V,E)$ is called {\em unit} if it has an interval representation with the required condition such that every interval has unit (or, same) length. We call it {\em unit-central-max-point tolerance graph} (in brief, unit central MPTG).  
\end{defn}

\begin{thm}\label{pi1}
Let $G$ be a simple undirected graph. Then the following are equivalent.

\begin{enumerate}

\item $G$ is a proper central MPTG.
\item $G$ is a unit central MPTG.
\item $G$ is a proper interval graph.

\end{enumerate}
\end{thm}

\begin{proof}
$\mathbf{(1\Longleftrightarrow 2):}$ Let $G$ be a proper central MPTG with respect to the representation $(I_{i},c_{i})$ where $I_{i}=[a_{i},b_{i}]$, $c_{i}$ be the center point of $I_{i}$ for each vertex $i\in V$. First we arrange the intervals according to increasing order of left end points. As no interval properly contains another, the right end points have the same order as left end points and so as the center points as well. We process the representation from left to right, adjusting all intervals to length $l$ where $l$ is the length of first interval (i.e; $|I_{1}|=l$). At each step until all intervals been adjusted $I_{x}$ be the leftmost unadjusted interval.

\noindent Let $I_{j}$ be an adjusted interval occurs before $I_{x}$, then one of following things happen.

\begin{enumerate}
\item[(1)] $c_{j}\notin I_{x}$.

\item[(2)] $c_{j}\in I_{x},c_{x}\in I_{j}$.

\item[(3)] $c_{j}\in I_{x}$ but $c_{x}\notin I_{j}$.
\end{enumerate}

\noindent Let $I_{j_{1}}$ and $I_{j_{2}}$ be any two adjusted intervals referred in $(2)$ and $(3)$ respectively. Then $I_{j_{2}}$ must occur before $I_{j_{1}}$ otherwise the right end point of $I_{j_{2}}$ would occur before $c_{x}$ and so $I_{j_{2}}$ would be properly contained in $I_{j_{1}}$ as $I_{j_{1}}$ contains $c_{x}$. But this is a contradiction.

\noindent Now if $I_{x}$ does not contain center point of any adjusted intervals then take 
$\alpha=a_{x}$. If $I_{x}$ contains center points of some adjusted intervals and $I_{i}$ be the leftmost among them, then $c_{l}\in I_{x}$ for all $i\leq l\leq x$ as all of them have same length. Now if $c_{x}\in I_{i}$, then take $\alpha=c_{i}$. It follows from the last paragraph that $c_{x}\in I_{l}$ for all $i\leq l\leq x$ in this case. Now if $c_{x}\notin I_{i}$, $c_{x}\notin I_{l}$ for any $l<i$. Let $I_{j}$ be the leftmost interval for which $c_{x}\in I_{j}$. Then $i<j<x$. Take $\alpha=c_{j}$ in this case. Now if no such $I_{j}$ exists between $I_{i}$ and $I_{x}$ i.e; if $c_{x}\notin I_{j}$ for all $i\leq j<x$ then $c_{x}\notin I_{l}$ for any $l<x$. Take $\alpha=b_{l}$ in this case where $b_{l}$ is the rightmost endpoint for which $c_{l}\in I_{x},c_{x}\notin I_{l}$. Clearly $i\leq l<x$. We adjust the portion $[a_{x},\infty)$ by shrinking or expanding $[a_{x},b_{x}]$ to $[\alpha,\alpha+l]$ and scaling and shifting $[b_{x},\infty)$ to $[\alpha+l,\infty)$. Iterating this operation produces the unit central MPTG representation.

\noindent Now it is sufficient to show adjusting $I_{x}$ in above way does not affect the adjacency
of vertex $x$ with previous intervals. When $\alpha=a_{x}$, then $\alpha=a_{x}>c_{l}$ for all $l<x$. Hence $c_{l}\notin I_{x}$ after adjustment. When $\alpha=c_{i}$, then $c_{i}=\alpha\in I_{x}$ and $c_{x}=\alpha+\dfrac{l}{2}=c_{i}+\dfrac{l}{2}=b_{i}\in I_{i}$. Moreover for all $i<l<x$, $c_{x}\in I_{l}$ and $\alpha=c_{i}<c_{l}<b_{i}=\alpha+\dfrac{l}{2}$ (i.e; $c_{l}\in I_{x}$) and $c_{x}\notin I_{l}$ for all $l<i$ after adjustment. When $\alpha=c_{j}$ then the arguments are similar as above.
Again when $\alpha=b_{l}$, then $c_{x}=\alpha+\dfrac{l}{2}=b_{l}+\dfrac{l}{2}>b_{l}$ which imply $c_{x}\notin I_{l}$. Hence $c_{x}\notin I_{k}$ for all $k<l$.

\vspace{.5em}\noindent
Conversely, if $G$ is a unit central MPTG then all intervals associated to the vertices of $G$ must be of the same length. Thus none of them contains other properly and so $G$ is a proper central MPTG with the same interval representation.

\vspace{.3em}
\noindent
$\mathbf{(3\Rightarrow 1):}$ Let $G=(V,E)$ be a proper interval graph. So the reduced graph $\hat{G}=(\hat{V},\hat{E})$ is an induced subgraph of $G(n,r)=(V_{n},E^{'})$ for some $n,r\in\mathbb{N}$ with $n>r$, where $V_{n}=\{v_{1},v_{2},\hdots,v_{n}\}$ and $v_{i}\leftrightarrow v_{j}$ if and only if 
$|i-j|\leq r$ by condition $4$ of Theorem \ref{proper1}. Let $\hat{V}=\{v_{i_{1}},v_{i_{2}},\hdots,v_{i_{m}}\}\subseteq V_{n}$. Now for each $u\in V$, define $p_{u}=i_{j}$ if $u$ is a copy of $v_{i_{j}}$ and $I_{u}=[p_{u}-r,p_{u}+r]$. Firstly all intervals $I_{u}$ are of same length $2r$ and so none of them properly contains other.

\noindent Next let $u,v\in V$. Suppose $p_{u}=i_{j}$ and $p_{v}=i_{k}$. Then $u$ is a copy of $v_{i_{j}}$ and $v$ is a copy of $v_{i_{k}}$. If $uv\in E$, then $v_{i_{j}}v_{i_{k}}\in \hat{E}\subseteq E^{'}$. Therefore $|i_{j}-i_{k}|\leq r \Rightarrow |p_{u}-p_{v}|\leq r \Rightarrow p_{v} \in I_{u}$ and $p_{u}\in I_{v} \Rightarrow p_{u},p_{v}\in I_{u}\cap I_{v}$. Finally, let $uv\notin E$. Then $v_{i_{j}}v_{i_{k}}\notin \hat{E}$. Since $\hat{G}$ is an induced subgraph of $G(n,r)$, we have $v_{i_{j}}v_{i_{k}}\notin E^{'}$. Then $|i_{j}-i_{k}|>r \Rightarrow |p_{u}-p_{v}|>r \Rightarrow p_{v}\notin I_{u}$ and $p_{u}\notin I_{v}$. Thus $G$ is a proper central MPTG.

\vspace{.3em}

\noindent 
$\mathbf{(1\Rightarrow 3):}$ let $G=(V,E)$ be a proper central MPTG with a proper central MPTG representation $(I_{u},c_{u})$ where $I_{u}=[l_{u},r_{u}]$, $c_{u}$ be the center point of $I_{u}$ for each $u\in V$. We arrange vertices according to the increasing order of center points, $V=\{v_{1},v_{2},\hdots,v_{n}\}$. To prove that G is a proper interval graph we show that vertices of $G$ satisfy condition $3$ of Theorem \ref{proper1} with respect to the above ordering. 

\noindent Denote $I_{u_{i}}=[l_{u_{i}},r_{u_{i}}]$ by $[l_{i},r_{i}]$ and $c_{i}=\dfrac{l_{i}+r_{i}}{2}$ for $i=1,2,\hdots,n$. Let $i<j<k$ and $u_{i}u_{k}\in E$. Then $c_{i}<c_{j}<c_{k}$. Now since G is a central MPTG, $c_{k}-c_{i}\leq \text{min}\{c_{i}-l_{i},c_{k}-l_{k}\}$. Now $c_{j}-c_{i}<c_{k}-c_{i}\leq c_{i}-l_{i}$. Now if $l_{j}>c_{i}$, then $l_{k}\leq c_{i}<l_{j}<c_{j}<c_{k}$ as $c_{k}-c_{i}\leq c_{k}-l_{k}$. So $[l_{j},c_{j}] \subsetneqq [l_{k},c_{k}]$. But this implies $[l_{j},r_{j}]\subsetneqq [l_{k},r_{k}]$ which contradicts the fact that G is a proper central MPTG. Thus $l_{j}\leq c_{i}$. So we have $c_{j}-c_{i}\leq c_{j}-l_{j}$. Hence $c_{j}-c_{i}\leq \text{min} \{c_{i}-l_{i},c_{j}-l_{j}\}$. Therefore $u_{i}u_{j}\in E$, as required. Similarly, it can be shown that $u_{j}u_{k} \in E$. Thus $G$ is a proper interval graph.
\end{proof}

\noindent Next following observation \ref{obsmptg} we have the following in a similar way:

\begin{prop}
Let $G$ be a simple undirected graph. Then $G$ is a central MPTG if and only if there exists a
central interval catch digraph $D$ such that $G = D\cap D^T$ where $D^T$ is the digraph obtained from D by reversing
direction of every arc.
\end{prop}

\noindent It is proved in \cite{Catanzaro} that if $G$ is an MPTG with non-adjacent vertices $u$ and $v$, then $G[N(u)\cap N(v)]$ is an interval graph. We found the following analogous result for central MPTG.
\begin{figure}
\begin{center}
\begin{tikzpicture}
\draw (0,-.5)--(0,0.5);
\draw (0,-.5)--(1,0-.5);
\draw (0,-0.5)--(-1,-0.5);
\draw (0,0.5)--(-1,-0.5);
\draw (0,0.5)--(1,-0.5);
\draw (0,-0.5)--(-0.5,-1.5);
\draw (-1,-0.5)--(-0.5,-1.5);
\draw(-1,-0.5)--(0.5,-1.5);
\draw (0,-0.5)--(0.5,-1.5);
\draw (1,-0.5)--(-0.5,-1.5);
\draw (0.5,-1.5)--(1,-0.5);

\draw  [fill=black](0,-0.5) circle [radius=0.05];
\draw  [fill=black](1,-0.5) circle [radius=0.05];
\draw  [fill=black](0,0.5) circle [radius=0.05];
\draw  [fill=black](-1,-0.5) circle [radius=0.05];
\draw  [fill=black](-0.5,-1.5) circle [radius=0.05];
\draw  [fill=black](0.5,-1.5) circle [radius=0.05];

\node [left] at (-1,-0.5) {$u$};
\node [right] at (1,-0.5) {$v$};
\end{tikzpicture}

\caption{The graph $G$ in Example \ref{notcmptg}}\label{figex311}
\end{center}
\end{figure}
\begin{prop}
If $G$ is a central MPTG with non-adjacent vertices $u$ and $v$, then $G[N(u)\cap N(v)]$ is a proper interval graph.\label{p1}
\end{prop}

\begin{proof}
Let $G=(V,E)$ be a central MPTG with an interval representation $\Set{I_{i}=[a_{i},b_{i}]}{i\in V}$ where vertices are arranged according to $C$-order. Let $c_{i}$ be the center point of $I_{i}$. Suppose $u<v$. Then vertices of $G[N(u)\cap N(v)]$ which occur between $u,v$ form a clique from (\ref{4p1}). Again if there exist a vertex of $G[N(u)\cap N(v)]$ occurs before $u$ then no vertex can occur after $v$ which belongs to $G[N(u)\cap N(v)]$ and conversely follows from (\ref{4p1}) and the fact $uv\notin E$. Moreover vertices of $G[N(u)\cap N(v)]$ that occur before $u$ form a clique. Let $x<y<u<v$ such that $x,y\in G[N(u)\cap N(v)]$ then $c_{x}<c_{y}<c_{u}<c_{v}\leq b_{x}$ as $vx\in E$. This implies $c_{y}\in[c_{x},b_{x}]\subset I_{x}$. Now if $a_{y}>c_{x}$, then $a_{u}\leq c_{x}<a_{y}<c_{y}<c_{u}$ (as $ux\in E$) which implies $c_{y}-a_{y}<c_{u}-a_{u}$. Hence $|I_{y}|<|I_{u}|$. Again as $a_{y},c_{y}\in[a_{u},c_{u}]$ from above we can conclude that $b_{y}\leq b_{u}$. But $yv\in E$ implies $a_{v}<c_{y}<c_{u}<c_{v}\leq b_{y}$ which imply $c_{u}\in[a_{v},c_{v}]\subset I_{v}$. Also as $u$ is not adjacent to $v$, $c_{v}>b_{u}$. Hence from above we get $b_{u}<c_{v}\leq b_{y}$ which is a contradiction. Therefore $a_{y}\leq c_{x}$. So $c_{x}\in [a_{y},c_{y}]\subset I_{y}$. Hence $xy\in E$. Similarly one can show vertices of $G[N(u)\cap N(v)]$ which occur after $v$ form a clique.

\vspace{0.5 em}\noindent 
Now let $\{u_{i}|u_{i}<u\}$ be the vertices of $G[N(u)\cap N(v)]$ arranged in $C$-order and $\{x_{j}|u<x_{j}<v\}$ be the vertices of $G[N(u)\cap N(v)]$ arranged according to increasing order of left end points. From above observations it is clear that $u_{i}\left[x_{j}\right]$ 's form clique for $u_{i}<u$ $\left[u<x_{j}<v\right]$. Let $u_{k}$ $\left[x_{m}\right]$ be the last vertices occurred before $u$ $\left[ \textrm{between} \ u \ \textrm{and}\  v\right]$ respectively. Now we show that $G[N(u)\cap N(v)]$ becomes a proper interval graph with respect to the ordering $\{u_{1},\hdots,u_{k},x_{1},\hdots,x_{m}\}$. In this ordering by $p \prec q$ we mean $p$ occurs before $q$. Infact we will show that the vertices satisfy condition $3$ of Theorem \ref{proper1} with respect to the ordering $\prec$. Let $u_{l}\prec x_{i} \prec x_{j}$ where $1\leq l \leq k,1\leq i,j\leq m $ such that $u_{l}x_{j}\in E$. Then $c_{u_l}<c_{u}<c_{x_{i}},c_{x_{j}}<c_{v}<b_{u_l}$ as $vu_{l}\in E$. This implies $c_{x_{i}}\in[c_{u_l},b_{u_l}]\subset I_{u_l}$. Now as $u_{l}x_{j}\in E$, $a_{x_{j}}\leq c_{u_l}<c_{u}<c_{x_{i}}$ implies $a_{x_{i}}<a_{x_{j}}\leq c_{u_l}<c_{x_{i}}$ (as $x_{i}\prec x_{j}\Longleftrightarrow a_{x_{i}}<a_{x_{j}}$) which imply $c_{u_l}\in [a_{x_{i}},c_{x_{i}}]\subset I_{x_{i}}$. Hence $u_{l}x_{i}\in E$. Let $u_{i}\prec u_{j}\prec x_{l}$ where $1\leq i,j\leq k,1\leq l\leq m$ such that $u_{i}x_{l}\in E$. Then $u_{i}\prec u_{j}\prec x_{l}\prec v$ clearly. Now from  (\ref{4p1}) one can conclude $u_{j}x_{l}\in E$. Similarly one can show if there exists vertices of $G[N(u)\cap N(v)]$ that occurs after $v$, then with respect to the ordering $\{x_{1},\hdots,x_{m},v_{1},\hdots,v_{k}\}$ (use $p\prec^{'}q$ if and only if $p$ occurs before $q$ in this ordering) $G[N(u)\cap N(v)]$ forms a proper interval graph where $\{x_{i}|u<x_{i}<v\}$ are vertices of $G[N(u)\cap N(v)]$ arranged according to increasing order of right end points, and $\{v_{j}|v<v_{j}\}$ are vertices of $G[N(u)\cap N(v)]$ arranged in $C$-order.
\end{proof}

\noindent
The above proposition leads to a construction of the following forbidden graph for the class of central MPTG.

\begin{exmp}\label{notcmptg}
By Proposition \ref{p1} we see that the graph $G$ (see Figure \ref{figex311}) formed by taking $K_{1,3}$ together with two non-adjacent vertices (say, $u,v$) which are adjacent to each vertex of $K_{1,3}$ is not a central MPTG.
\end{exmp}

\section{$50\%$ max-tolerance graphs}

\noindent In this section we introduce a new type of max-tolerance graph, namely {\em $50\%$ max-tolerance graph} similar to the concept of $50\%$ min-tolerance graph \cite{Bogart}.

\begin{defn}\label{c5}
A max-tolerance graph $G=(V,E)$ is a {\em $50\%$ max-tolerance graph} if $t_{u}=\dfrac{|I_{u}|}{2}$ for all $u\in V$ where $t_{u}$ denotes the tolerance associated with the vertex $u\in V$.
\end{defn}

\noindent  We know that classes of unit (min) tolerance graphs and $50\%$ (min) tolerance graphs are same \cite{Bogart}. So the very natural question which comes to our mind that whether unit and $50\%$ max-tolerance graphs denote the same graph class. We settle down this question in the following theorem.

\begin{thm}\label{com}
Central MPTG (i.e., unit max-tolerance graphs) and the class of $50\%$ max-tolerance graphs are not comparable.
\end{thm}

\begin{proof}
The graph $K_{1,n}$ is an interval graph for any natural number $n$. So by Theorem \ref{sp} it is a central MPTG and hence a unit max-tolerance graph as well. In Lemma \ref{px4} we will show $K_{1,n}, n\geq 8$ is not a $50\%$ max-tolerance graph. Again from Theorem \ref{c61} one can verify that $\overline{C_{6}}$ is a $50\%$ max-tolerance graph but it is not a central MPTG. Thus $K_{1,n},n\geq 8$ and $\overline{C_{6}}$ separate unit max-tolerance graph from $50\%$-max-tolerance graph. Hence these graph classes become incomparable.    
\end{proof}

\noindent We do the proof verification of the above theorem with the help of following Lemma.

\begin{lem} \label{px4}
\noindent $K_{1,n}$ where $n\geq 8$ is not a $50\%$ max-tolerance graph.
\end{lem}

\begin{proof}
\noindent
On the contrary lets assume that $K_{1,n}, n\geq 8$ has a $50\%$ max-tolerance representation. By suitable scaling and shifting origin without loss of generality, we may assume that $I_{u}=[0,1]$, interval corresponding to central vertex $u$ and $I_{v_{i}}=[a_{i},b_{i}]$, intervals corresponding to the pendant vertices $v_{i}$ where $i\in\{1,\hdots,n\}$.

\noindent We consider the intervals corresponding to any two pendants must be distinct as they are non-adjacent.
Let $[a,b]$ be an interval corresponding to a pendant vertex $I_{v}$.  Then $I_{v}$ must belong to one of the four sets, $S_{1}=\{I_{v}|a\leq 0<1\leq b\}$, $S_{2}=\{I_{v}|a<0,0<b< 1\}$, $S_{3}=\{I_{v}|0<a<1,b>1\}$, $S_{4}=\{I_{v}|0\leq a< b\leq 1\}$. We note that any pendant vertex cannot have $[0,1]$ as its interval representation. Hence the above sets are mutually exclusive. As every pendant vertex is adjacent to the central vertex $u$ hence it follows that the four sets are also exhaustive. 
Now we show that
\begin{eqnarray} \label{50}
|S_{1}|\leq 1 \hspace*{.7em}\text{and} \hspace*{.7em} |S_{i}|\leq 2 \hspace*{.7em} \text{for} \hspace*{.7em} i\in\{2,3,4\}.
\end{eqnarray}

\begin{itemize}
\item $|S_{1}|\leq 1$.

\vspace{.1em}\noindent
Let $[a_{1},b_{1}],[a_{2},b_{2}]$ be the two intervals associated to two such pendant vertices
(say $v_{1},v_{2}$). Then the  $[a_{i},b_{i}] \supseteq [0,1]$ for $i=\{1,2\}.$ Now $[a_{i},b_{i}]\cap[0,1]=[0,1]$ clearly. As every pendant vertex is adjacent to $u$ and $I_{u}$ is of unit length, so each pendant vertex has tolerance at most $1$ and therefore has length at most $2$. But $[a_{1},b_{1}]\cap [a_{2},b_{2}]\supseteq [0,1]\Rightarrow |[a_{1},b_{1}]\cap [a_{2},b_{2}]|\geq 1$. Hence $v_{1}\leftrightarrow v_{2}$ which contradicts the fact that they are pendant vertices.

\item $|S_{2}|\leq 2$.

\vspace{.1em}
\noindent
\noindent Let $[a_{1},b_{1}],[a_{2},b_{2}]$ and $[a_{3},b_{3}]$ be three such intervals corresponding to the vertices (say $v_{1},v_{2},v_{3}$). First we will show that for any two of the three intervals, one must contain the other. On contrary, without loss of generality assume that $[a_{1},b_{1}],[a_{2},b_{2}]$ are two such intervals which are not containing each other and let $a_{1}<a_{2}$, $b_{1}<b_{2}$. It is sufficient to show for one case. Now let
$[a_{1},b_{1}]\cap [a_{2},b_{2}]=[a_{2},b_{1}]$. Now as $v_{1}\leftrightarrow u$, $|[a_{1},b_{1}]\cap[0,1]|\geq \text{max} \{\dfrac{1}{2},\dfrac{b_{1}-a_{1}}{2}\} \Rightarrow b_{1}\geq \dfrac{1}{2},\dfrac{b_{1}-a_{1}}{2}.$ This implies \begin{eqnarray}
\label{px1}
b_{1}\geq -a_{1}
\end{eqnarray}
Similarly $b_{2}\geq -a_{2}$.

\noindent Now $-a_{1}\leq b_{1}\Rightarrow -a_{1}+b_{1}\leq 2b_{1}\Rightarrow -a_{1}+b_{1}< 2b_{1}-2a_{2}\Rightarrow \dfrac{-a_{1}+b_{1}}{2}< b_{1}-a_{2}$. But $v_{1}\nleftrightarrow v_{2}$. Hence $b_{1}-a_{2}<\dfrac{b_{2}-a_{2}}{2}\Rightarrow 2b_{1}-2a_{2}<b_{2}-a_{2}\Rightarrow 2b_{1}-a_{2}<b_{2}.$ Hence $b_{2}>2b_{1}-a_{2}>2b_{1}\geq 1$ (Since $b_{1}\geq \dfrac{1}{2}$ and $a_{2}<0$) which contradicts the fact $b_{2}<1.$ Hence the three intervals form a well ordered set with inclusion as the ordering. \\
From the statement above without loss of generality we can assume that $[a_{1},b_{1}] \supseteq [a_{2},b_{2}]\supseteq [a_{3},b_{3}]$.
\vspace{.1 cm}
\noindent Now $[a_{1},b_{1}]\cap [a_{2},b_{2}]=[a_{2},b_{2}]$. Therefore $2(b_{2}-a_{2})<b_{1}-a_{1}$. Similarly $b_{2}-a_{2}>2(b_{3}-a_{3}).$ Hence $b_{1}-a_{1}>4(b_{3}-a_{3})>4b_{3}\geq 2$ (since $b_{3}\geq\dfrac{1}{2}$ as $v_{3}\leftrightarrow u$). Now $b_{1}-a_{1}\leq 2b_{1}< 2$ from (\ref{px1}) and the fact $b_{1}<1.$ Hence we are through.

\item $|S_{3}|\leq 2$.

\noindent This proof is same as the last case.

\item $|S_{4}|\leq 2$.

\noindent On the contrary we assume that there exists three such vertices with representations  $[a_{i},b_{i}]\subseteq [0,1]$ for $i\in \{1,2,3\}$.

\begin{enumerate}
\item First we will show that for any two intervals, one must \textbf{not} contain the other.
\vspace{0.1em}

If not, we assume $[a_{i},b_{i}]\subseteq [a_{j},b_{j}]$ for some $i,j\in\{1,2,3\}$. Then  $v_{i}\leftrightarrow u\Rightarrow b_{i}-a_{i}\geq \dfrac{1}{2}$. Hence $|[a_{i},b_{i}]\cap[a_{j},b_{j}]|=|[a_{i},b_{i}]|\geq\dfrac{1}{2}$. Again $v_{i}\nleftrightarrow v_{j}\Rightarrow b_{i}-a_{i}<\dfrac{b_{j}-a_{j}}{2}\leq \dfrac{1}{2}$ which is a contradiction.

\vspace{0.1em}

\item Next we will show that no two intervals are disjoint. 

\vspace{0.1em}

If possible let $[a_{1},b_{1}]$ and $[a_{2},b_{2}]$ be disjoint. Without loss of generality we can assume $a_{1}< a_{2}.$ Hence since they are disjoint $ b_{1}< a_{2}.$ But $b_{1}\geq \dfrac{1}{2}$ as $v_{1}\leftrightarrow u$. Hence $a_{2}>\dfrac{1}{2}$. This implies $b_{2}-a_{2}<\dfrac{1}{2}$ which contradicts $v_{2}\leftrightarrow u$.
\end{enumerate}

\noindent Now without loss of generality we can assume $a_{1}<a_{2}<a_{3}.$ Hence from $1$ and $2$ we conclude $a_{1}<a_{2}<a_{3}\leq b_{1}< b_{2}< b_{3}$. Under this situation we will show that there exist no choice of three such intervals. For this we first show $b_{2}>\dfrac{3}{4}$.
If $a_{2}\leq \dfrac{a_{1}+b_{1}}{2}$ then $ b_{1}-a_{2}\geq b_{1}-\dfrac{a_{1}+b_{1}}{2}=\dfrac{b_{1}-a_{1}}{2}.$ Then $b_{1}-a_{2}<\dfrac{b_{2}-a_{2}}{2}$ since $v_{1}\nleftrightarrow v_{2}$.
Hence $b_{2}>2b_{1}-a_{2}\geq 2b_{1}-\dfrac{a_{1}+b_{1}}{2}=b_{1}+\dfrac{b_{1}-a_{1}}{2}\geq \dfrac{1}{2}+\dfrac{1/2}{2}=\dfrac{3}{4}$ (Since $v_{1}\leftrightarrow u$, $b_{1}-a_{1}\geq \dfrac{1}{2}$ and hence $b_{1}\geq \dfrac{1}{2}$).
For other case ,i.e., $a_{2}> \dfrac{a_{1}+b_{1}}{2}=\dfrac{b_{1}-a_{1}}{2}+a_{1}\geq \dfrac{b_{1}-a_{1}}{2}\geq \dfrac{1/2}{2}=\dfrac{1}{4}.$ Now $b_{2}-a_{2}\geq \dfrac{1}{2}\Rightarrow b_{2}\geq a_{2}+\dfrac{1}{2}>\dfrac{1}{4}+\dfrac{1}{2}=\dfrac{3}{4}\Rightarrow b_{2}>\dfrac{3}{4}.$ 

\vspace{0.6em}

\noindent For the remaining of the proof we split into two cases.

\begin{enumerate}
\item If $a_{3}\leq \dfrac{a_{2}+b_{2}}{2}$ then $b_{2}-a_{3}\geq b_{2}-\dfrac{a_{2}+b_{2}}{2}=\dfrac{b_{2}-a_{2}}{2}.$ Then $b_{2}-a_{3}<\dfrac{b_{3}-a_{3}}{2}$ as $v_{2}\nleftrightarrow v_{3}.$ This implies $b_{3}>2b_{2}-a_{3}\geq 2b_{2}-\dfrac{a_{2}+b_{2}}{2}=b_{2}+\dfrac{b_{2}-a_{2}}{2}> \dfrac{3}{4}+\dfrac{1/2}{2}=1.$ Hence we arrive at a contradiction.

\item If $a_{3}>\dfrac{a_{2}+b_{2}}{2}$ then since $b_{3}-a_{3}\geq \dfrac{1}{2}$ (as $v_{3}\leftrightarrow u$), we have

\begin{eqnarray}
\label{px2}
 b_{3}> \dfrac{a_{2}+ b_{2}}{2}+ \dfrac{1}{2}
\end{eqnarray}

Now if $a_{2}\leq \dfrac{a_{1}+b_{1}}{2}$ then $b_{2}> 2b_{1}-a_{2}$ as before. Hence $\dfrac{b_{2}+a_{2}}{2}>b_{1}.$  Hence using (\ref{px2}) we get $b_{3}>b_{1}+\dfrac{1}{2}\geq \dfrac{1}{2}+ \dfrac{1}{2}=1$ contradicting $b_{3}\leq 1.$ 

\vspace{0.2em}

Now if $a_{2}> \dfrac{a_{1}+b_{1}}{2}$ then $b_{2}>\dfrac{a_{1}+b_{1}}{2}+\dfrac{1}{2}$ since $b_{2}-a_{2}\geq \dfrac{1}{2}$. Hence $a_{2}+b_{2}>a_{1}+b_{1}+\dfrac{1}{2}.$ Again using (\ref{px2}) we get $b_{3}>\dfrac{a_{1}+b_{1}}{2}+\dfrac{1}{4}+\dfrac{1}{2}\geq \dfrac{1/2}{2}+ \dfrac{1}{4}+ \dfrac{1}{2}=1$ (as $b_{1}\geq \dfrac{1}{2},a_{1}\geq 0$). Hence we again arrive at a contradiction.
\end{enumerate}

\noindent Hence we have established our claim.
\end{itemize}

Hence the maximum number of possible pendant vertices is $1+2+2+2=7.$ So we are done if we take atleast 8 pendant vertices, i.e., $n\geq 8.$
\end{proof}

\noindent
In the following Propositions we will show that the class of proper interval graphs and any cycle of length $n$ both are subclasses of $50\%$ max-tolerance graphs.

\begin{prop}\label{proper2}
 A proper interval graph is a $50\%$ max-tolerance graph.
\end{prop}

\begin{proof}
Let $G=(V,E)$ be a proper interval graph. Now as from Theorem \ref{pi1} it follows that proper interval graphs are same as proper central MPTG, we can take $(I_{u},c_{u})$ to be a proper central MPTG representation of $G$ where $I_{u}=[a_{u},b_{u}]$, $c_{u}$ denotes the center point of $I_{u}$ for each $u\in V$. We will show that $G$ is a $50\%$ max-tolerance graph with respect to the same interval representation. Suppose $uv\in E$. Then $I_{u}\cap I_{v}\supseteq \{c_{u},c_{v}\}$. This implies $|I_{u}\cap I_{v}|\geq \dfrac{|I_{u}|}{2}, \dfrac{|I_{v}|}{2}$ as none of them contains other properly. Hence $|I_{u}\cap I_{v}|\geq \dfrac{1}{2} \text{max}\{|I_{u}|,|I_{v}|\}$ which imply $uv\in E$ in its $50\%$ max-tolerance representation. Again $uv\notin E$ implies $c_{u}\notin I_{v}$ or $c_{v}\notin I_{u}$. If $c_{u}\leq c_{v}$ then $c_{u}\notin I_{v} \Rightarrow a_{v}>c_{u}\Rightarrow |I_{u}\cap I_{v}|<\dfrac{|I_{u}|}{2}$. Also $c_{v}\notin I_{u}\Rightarrow c_{v}>b_{u}\Rightarrow |I_{u}\cap I_{v}|<\dfrac{|I_{v}|}{2}$. From this we can conclude that $|I_{u}\cap I_{v}|<\dfrac{1}{2} \text{max}\{|I_{u}|,|I_{v}|\}$. Hence $uv\notin E$ in its $50\%$ max-tolerance representation.
\end{proof}

\begin{prop}\label{c60}
Any cycle $C_{n},(n\geq 3)$ is a $50\%$ max-tolerance graph.
\end{prop}

\begin{proof}
We prove $C_{n}$ to be a $50\%$ max-tolerance graph by constructing interval representation $\{I_{v_{i}}=[a_{i},b_{i}]|v_{i}\in V\}$ where the vertex set $V$ is labelled clockwise starting in an arbitary vertex.
Let $c_{i}$ be the center point of $I_{v_{i}}$ and tolerance $t_{v_{i}}=\dfrac{|I_{v_{i}}|}{2}$. 
 
\begin{itemize}
\item Let $n=3$. We associate intervals $I_{v_{1}}=I_{v_{2}}=I_{v_{3}}=[1,2]$.
\item Let $n=4$. We associate intervals $I_{v_{1}}=[1,4.6],I_{v_{2}}=[2,4],I_{v_{3}}=[2.9,4.9],I_{v_{4}}=[2.7,6.3]$.
\item Let $n=5$. We associate intervals $I_{v_{1}}=[10,30],I_{v_{2}}=[16,28],I_{v_{3}}=[18,24],I_{v_{4}}=[15,21],I_{v_{5}}=[9,21]$.
\item Let $n\geq 6$. We prove now in two cases considering $n$ even and odd. We define
$$k=\dfrac{n}{2}\ \text{ when } n \text{ is even},\quad \text{and}\quad k=\dfrac{n+1}{2}\ \text{ when }n \text{ is odd.}$$

\noindent When \textbf{\emph{$n$ is even}} we define $I_{v_{1}}=[2-n,n], I_{v_{2}}=[1,n], I_{v_{i}}=[1,1+\dfrac{k-1}{2^{i-4}}]$ for $i\in\{3,\hdots,k\}$, $I_{v_{k+1}}=[1-\dfrac{k-1}{2^{k-3}},1+\dfrac{k-1}{2^{k-3}}]$, $I_{v_{j}}=[1-\dfrac{k-1}{2^{n-j-2}},1]$ for $j\in\{k+2,\hdots,n\}$.

\noindent We verify that the above representation indeed give $50\%$ max-tolerance representation of $C_{n},n\geq 6$.

\noindent\textbf{Claim $1$: \textit{We will show that $v_{1}$ is adjacent to only $v_{2},v_{n}$}}.

\vspace{.2em}\noindent We note $|I_{v_{1}}\cap I_{v_{2}}|=|I_{v_{2}}|=n-1=\dfrac{|I_{v_{1}}|}{2}$. This implies $v_{1}\leftrightarrow v_{2}$. Also since $n\geq 6$, $5-2n<2-n$. Hence $a_{n}<a_{1}$. Also $b_{n}=1<n=b_{1}$. So $|I_{v_{1}}\cap I_{v_{n}}|=|[a_{1},b_{n}]|=|[2-n,1]|=n-1=\dfrac{|I_{v_{1}}|}{2}>n-2=\dfrac{|I_{v_{n}}|}{2}$. This implies $v_{1}\leftrightarrow v_{n}$. Next we  show $v_{1}\nleftrightarrow v_{i}$ for $i\in\{3,\hdots,k\}$. Clearly $2-n<1$ implies $a_{1}<a_{i}$. Also $n-2<(n-1)2^{i-3} $ for $i\geq 3$. Hence
\begin{equation}\label{eq1}
1+\dfrac{k-1}{2^{i-4}}<n
\end{equation}
Therefore $b_{i}<b_{1}$. Hence $|I_{v_{1}}\cap I_{v_{i}}|=|I_{v_{i}}|=
\dfrac{k-1}{2^{i-4}}<n-1 \text{ from } (\ref{eq1})=\dfrac{|I_{v_{1}}|}{2}$. Hence $v_{1}\nleftrightarrow v_{i}$. Analogously $v_{1}\nleftrightarrow v_{j}$ for $j\in\{k+2,\hdots,n-1\}$.
Again $n-2<(n-1)2^{k-2}$ implies
\begin{equation}\label{eq2}
1+\dfrac{k-1}{2^{k-3}}<n
\end{equation}
Hence $a_{k+1}=1-\dfrac{k-1}{2^{k-3}}>2-n=a_{1}.$
Similarly, $b_{k+1}= 1+\dfrac{k-1}{2^{k-3}}< n=b_{1}.$ Hence
$|I_{v_{1}}\cap I_{v_{k+1}}|=|I_{v_{k+1}}|=
\dfrac{k-1}{2^{k-4}}=\dfrac{n-2}
{2^{k-3}}<n-1=\dfrac{|I_{v_{1}}|}{2}.$ Hence $v_{1}\nleftrightarrow v_{k+1}$.

\vspace{1em}
\noindent \textbf{Claim $2$: \textit{We will show that $v_{2}$ is adjacent to only $v_{1},v_{3}$}}.

 \noindent  From previous case $v_{2}\leftrightarrow v_{1}$. Note that $|I_{v_{2}}\cap I_{v_{3}}|= |[1,n-1]|=n-2>\dfrac{n-1}{2}=\dfrac{|I_{v_{2}}|}{2}, \dfrac{n-2}{2}=\dfrac{|I_{v_{3}}|}{2}$. Hence $v_{2}\leftrightarrow v_{3}$. Next to show $v_{2}\nleftrightarrow v_{i}$ for $i\in\{4,\hdots,k\}$ we note, $a_{2}=a_{i}=1$ and $b_{i}<b_{1}[\text{ from(\ref{eq1}) }]=b_{2}$ for $i\in\{4,\hdots,k\}.$ Hence $|I_{v_{2}}\cap I_{v_{i}}|= |I_{v_{i}}|=\dfrac{k-1}{2^{i-4}}=\dfrac{n-2}{2^{i-3}}<\dfrac{n-1}{2}=\dfrac{|I_{v_{2}}|}{2}.$ This implies $v_{2}\nleftrightarrow v_{i}$ for $i\in\{4,\hdots,k\}.$ Again $|I_{v_{2}}\cap I_{v_{j}}|=\{1\}<\dfrac{n-1}{2}=\dfrac{|I_{v_{2}}|}{2}$. Hence $v_{2}\nleftrightarrow v_{j}$ for $j\in\{k+2,\hdots,n\}.$\\
Now clearly $a_{k+1}<1= a_{2}$ and $b_{k+1}<n=b_{2}$ from $( \ref{eq2} )$. Hence, $|I_{v_{2}}\cap I_{v_{k+1}}|= |[a_{2},b_{k+1}]|= \dfrac{k-1}{2^{k-3}}=\dfrac{n-2}{2^{k-2}}< \dfrac{n-1}{2}=\dfrac{|I_{v_{2}}|}{2}.$
Hence $v_{2}\nleftrightarrow v_{k+1}.$

\vspace{1em}
\noindent \textbf{Claim $3$:\textit{ Vertex $v_{i}$ where $i\in\{3,\hdots,k\}$ is adjacent to only $v_{i-1},v_{i+1}$}}.

\vspace{.1em} \noindent From previous case $v_{3}\leftrightarrow v_{2}$. Let $v_{i},v_{i^{'}}$ be two vertices where $i,i^{'}\in\{3,\hdots,k\}$. Let $i<i^{'}$. Then clearly $b_{i}>b_{i^{'}}$. Hence $|I_{v_{i}}\cap I_{v_{i^{'}}}|=|I_{v_{i^{'}}}|=\dfrac{k-1}{2^{i^{'}-4}}$. Now if $i^{'}=i+1$, then $\dfrac{k-1}{2^{i^{'}-4}}=\dfrac{k-1}{2^{i-3}}=\dfrac{|I_{v_{i}}|}{2}>\dfrac{|I_{v_{i^{'}}}|}{2}$ clearly. Hence $v_{i}\leftrightarrow v_{i^{'}}$ in this case. For $i^{'}>i+1$, $|I_{v_{i}}\cap I_{v_{i^{'}}}|=\dfrac{k-1}{2^{i^{'}-4}}<\dfrac{k-1}{2^{i-3}}=\dfrac{|I_{v_{i}}|}{2}$. Hence $v_{i}\nleftrightarrow v_{i^{'}}$. Also one can check trivially that $a_{k+1}<1=c_{k+1}=a_{i}$ where $i\in\{3,\hdots,k\}$. Moreover $|I_{v_{k}}\cap I_{v_{k+1}}|=|[c_{k+1},b_{k+1}]|=|[1,1+\dfrac{k-1}{2^{k-3}}]|=\dfrac{k-1}{2^{k-3}}=\dfrac{|I_{v_{k}}|}{2}$. Hence $v_{k}\leftrightarrow v_{k+1}$. Again $a_{i}=c_{k+1}=1$ and $b_{k+1}=1+\dfrac{k-1}{2^{k-3}}=c_{k}<b_{k}=1+\dfrac{k-1}{2^{k-4}}<1+\dfrac{k-1}{2^{i-4}}=b_{i}$ for $3\leq i<k$. Hence $|I_{v_{i}}\cap I_{v_{k+1}}|=\dfrac{k-1}{2^{k-3}}<\dfrac{k-1}{2^{i-3}}=\dfrac{|I_{v_{i}}|}{2}$. This shows $v_{i}\nleftrightarrow v_{k+1}$ for $3\leq i<k$. Now $|I_{v_{i}}\cap I_{v_{j}}|=\{1\}$ for $i\in\{3,\hdots,k\},j\in\{k+2,\hdots,n\}$. Hence $v_{i}\nleftrightarrow v_{j}$.

\vspace{1em}
\noindent \textbf{Claim $4$:\textit{ Vertex $v_{k+1}$ is adjacent to only $v_{k},v_{k+2}$}}.

\vspace{1em}
\noindent \textbf{Claim $5$:\textit{ Vertex $v_{i}$ where $i\in\{k+2,\hdots,n\}$ is adjacent to only $v_{i-1},v_{i+1}$}}.

\vspace{.4em}\noindent
Claim $4,5$ can be shown similarly to the previous cases.

\noindent When \textbf{\emph{$n$ is odd}} we define $I_{v_{1}}=[2-n,n],I_{v_{2}}=[1,n], I_{v_{i}}=[1,1+\dfrac{n-2}{2^{i-3}}]$ for $i\in\{3,\hdots,k\}$, $I_{v_{k+1}}=[1-\dfrac{3n-6}{2^{k-1}},1+\dfrac{3n-6}{2^{k-1}}]$,
$I_{v_{j}}=[1-\dfrac{3n-6}{2^{n+1-j}},1]$ for $j\in\{k+2,\hdots,n\}$. 
\end{itemize}
\hspace{1.8 em} \noindent This proof is analogous to the case when $n$ is even. 
\end{proof}

\begin{rmk}
{\em Combining Proposition \ref{proper2} and Proposition \ref{c60} along with results of \cite{Soto} we can easily conclude that proper intervals graphs and $C_{n}$ both belong to the intersection class of central MPTG and $50\%$ max-tolerance graphs}.
\end{rmk}

\section{Central interval catch digraphs}

\noindent
In 1984, Maehera posed the following conjecture:

\noindent \textbf{Maehera's conjecture}\cite{Maehera}: \label{conj1} If a digraph $G$ has no induced subdigraph isomorphic to one of the digraphs in Figure \ref{mhfig} and $A^*(G)$ satisfy the consecutive $1$'s property for rows, then $G$ is a central ICD.

\begin{figure}
	\centering
	
	\includegraphics[width=4.4in]{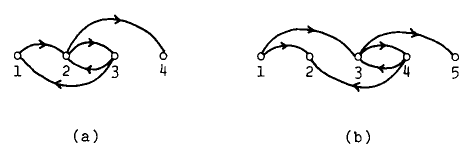}
	\caption{Maehara's forbidden digraphs for central ICD}\label{mhfig}    
\end{figure}

\noindent 
We show the digraphs $G_1,G_2,G_3$ in Example \ref{notcicd} disprove the conjecture. To see this we obtain the following necessary condition of a central ICD.

\begin{prop}\label{lcicd}
Let $G=(V,E)$ be a central ICD. Then there is an ordering of vertices $V=\set{v_1,v_2,\ldots,v_n}$ which satisfies (\ref{icd1}) and the following condition:
\begin{equation}\label{cicd2}
\text{for any }i<j,\text{ either }i_1\leqslant j_1\text{ or, }i_2\leqslant j_2,
\end{equation}
where $i_1$ and $i_2$ be the least and the highest numbers such that $i_1=i$ or, $v_iv_{i_1}\in E$ and $i_2=i$ or, $v_iv_{i_2}\in E$ for each $i=1,2,\ldots,n$. 
\end{prop}

\begin{proof}
Let $G=(V,E)$ be represented by $\{(I_{i},c_{i})|i\in V\}$ where $I_{i}=[a_{i},b_{i}]$ is an interval and $c_{i}$ be its center point. We arrange vertices of $G$ according to the increasing order of these center points. It is easy to check $G$ satisfy (\ref{icd1}) with respect to this ordering. On contrary lets assume for some $i<j$, $i_{1}>j_{1}$ and $i_{2}>j_{2}$. Then $j_1<i_1\leqslant i<j\leqslant j_2<i_2$. Now $v_{j}v_{j_{1}}\in E, v_{i}v_{j_{1}}\notin E$ imply $c_{j_1}\in I_j$ and $c_{j_1}\notin I_i$. Since $c_{j_1}\notin [a_i,b_i]$, either $c_{j_1}<a_i$ or $c_{j_1}>b_i$. If $c_{j_1}>b_i$, then $c_{i_1}\geq c_{j_1}>b_i$ as $i_1>j_1$. But then $v_iv_{i_1}\notin E$ which is a contradiction. Thus 
$a_{j}\leq c_{j_{1}}<a_{i}\leq c_{i}\leq c_{j}$. Hence $[a_{i},c_{i}]\subset [a_{j},c_{j}]$. Now as $v_{i}v_{i_{2}}\in E$, $c_{i_{2}}\in I_{i}\subset I_j$ which implies $v_{j}v_{i_{2}}\in E$. But then $i_2\leq j_2$ which is again a contradiction. Hence the proof follows.
\end{proof}

\begin{figure}[h]
	\begin{tikzpicture}
	
	\draw (-2,0)--(0,0);
	\draw (-2,0)--(-1,-1);
	\draw (-1,-1)--(0,0);
	\draw (-2,0)--(-1,1);
	\draw (0,0)--(-1,1);
	
	\draw [fill=black](-2,0) circle [radius=0.05];
	\draw [fill=black](0,0) circle [radius=0.05];
	\draw [fill=black](-1,-1) circle [radius=0.05];
	\draw [fill=black](-1,1) circle [radius=0.05];
	
	\node [left] at (-2,0) { $v_{2}$};
	\node [right] at (0,0) { $v_{3}$};
	\node [above] at (-1,1) { $v_{1}$};
	\node [below] at (-1,-1) { $v_{4}$};
	
	\draw [->](-2,0)--(-1,0);
	\draw [->](0,0)--(-1,0);
	\draw [->](-1,1)--(-1.5,.5);
	\draw [->](0,0)--(-.5,.5);
	\draw [->](-2,0)--(-1.5,-0.5);
	\draw[->] (-1,-1)--(-.5,-0.5);

\node[below] at (-1.2,-1.8) {{\footnotesize $G_{1}$ $(a=b=0,c=1)$}};

	\draw (1,0)--(3,0);
	\draw (1,0)--(2,-1);
	\draw (3,0)--(2,-1);
	\draw (3,0)--(2,1);
	\draw (1,0)--(2,1);
	
	\draw [fill=black](1,0) circle [radius=0.05];
	\draw [fill=black](3,0) circle [radius=0.05];
	\draw [fill=black](2,-1) circle [radius=0.05];
	\draw [fill=black](2,1) circle [radius=0.05];
	
	\node [left] at (1,0) { $v_{2}$};
	\node [right] at (3,0) { $v_{3}$};
	\node [above] at (2,1) { $v_{1}$};
	\node [below] at (2,-1) { $v_{4}$};

	\draw [->](1,0)--(2,0);
	\draw [->](3,0)--(2,0);
	\draw [->](2,1)--(1.5,.5);
	\draw [->](3,0)--(2.5,0.5);
	\draw [->](1,0)--(1.5,-0.5);
	\draw [->](2,-1)--(1.5,-0.5);
	\draw[->] (2,-1)--(2.5,-0.5);

\node[below] at (2.1,-1.8) {{\footnotesize $G_{2}$ $(a=0,b=c=1)$}};
	
	\draw (4,0)--(6,0);
	\draw (6,0)--(5,-1);
	\draw (4,0)--(5,-1);
	\draw (4,0)--(5,1);
	\draw (6,0)--(5,1);
	\draw (5,1)--(5,-1);
	
	\draw [fill=black](4,0) circle [radius=0.05];
	\draw [fill=black](6,0) circle [radius=0.05];
	\draw [fill=black](5,-1) circle [radius=0.05];
	\draw [fill=black](5,1) circle [radius=0.05];
	
	\node [left] at (4,0) { $v_{2}$};
	\node [right] at (6,0) { $v_{3}$};
	\node [above] at (5,1) { $v_{1}$};
	\node [below] at (5,-1) { $v_{4}$};

	\draw [->](4,0)--(4.7,0);
	\draw [->](6,0)--(5.3,0);
	\draw [->](5,1)--(4.5,.5);
	\draw [->](6,0)--(5.5,.5);
	\draw [->](5,-1)--(4.5,-0.5);
	\draw [->](4,0)--(4.5,-0.5);	
	\draw[->] (5,-1)--(5.5,-0.5);
	\draw[->](5,-1)--(5,.7);

\node[below] at (5.3,-1.8) {{\footnotesize $G_{3}$ $(a=b=c=1)$}};

	\draw (7.5,0)--(8,1);
	\draw (6.7,0)--(7.5,0);
	\draw (7.5,0)--(8.5,0);
	\draw (8,1)--(8.5,0);

	\draw [fill=black](7.5,0) circle [radius=0.05];
	\draw [fill=black](6.7,0) circle [radius=0.05];
	\draw [fill=black](8.5,0) circle [radius=0.05];
	\draw [fill=black](8,1) circle [radius=0.05];
	
	\node [above] at (8,1) { $v_{1}$};
	\node [below] at (8.5,-.3) { $v_{3}$};
	\node [below] at (7.5,-.3) { $v_{2}$};
	\node [below] at (6.7,-.3) { $v_{4}$};

	\draw [->](7.5,0)--(7,0);
	\draw [->](7.5,0)--(8,0);
	\draw [->](8,1)--(7.75,.5);
	\draw [->](8.5,0)--(8.25,.5);
	\draw [->](8.5,0)--(8,0);
	
	\node[below] at (8.5,-1.8) {{\footnotesize $G_{4}$ $(a=b=c=0)$}};
	\end{tikzpicture}\quad\quad\quad\null
	
	\vspace{-11em}
	\null\hspace{4.9in} 
	$\begin{array}{cccccc}
	&v_{1}&v_{2}& v_{3} &v_{4}\\
	\cline{2-5}
	\multicolumn{1}{c@{\,\vline}}{v_{1}} & 1 & 1 & 0 & \multicolumn{1}{c@{\,\vline}}{0}\\
	\multicolumn{1}{c@{\,\vline}}{v_{2}} & 0 & 1 & 1 & \multicolumn{1}{c@{\,\vline}}{1}\\
	\multicolumn{1}{c@{\,\vline}}{v_{3}} & 1 & 1 & 1 & \multicolumn{1}{c@{\,\vline}}{0} \\
	\multicolumn{1}{c@{\,\vline}}{v_{4}} & a & b & c & \multicolumn{1}{c@{\,\vline}}{1}\\
	\cline{2-5}
	\end{array}$
	
	\vspace{3em}
	\caption{Forbidden digraphs for central ICD when $|V|=4$ and their augmented adjacency matrix.}\label{forbid11}		
\end{figure}	

\begin{exmp}\label{notcicd}
Consider the digraphs $G_k=(V_k,E_k)$ with vertex set $V_k$ and edge set $E_{k}$ for $k=1,2,3,4$ in Figure \ref{forbid11}. Note that $G_4$ is the digraph {\tt (a)} in Figure \ref{mhfig}. From Theorem \ref{icd1} one can easily verify that $v_1<v_2<v_3<v_4$ and its reverse ordering are the only possible ICD ordering of $V_k$ for each graph $G_k$ (i.e., for which $A^*(G_k)$ satisfies the consecutive $1$'s property for rows). Now this augmented adjacency matrix (in Figure \ref{forbid11}) shows a contradiction to (\ref{cicd2}) for $i=2$, $j=3$ for each $k$. Thus these graphs are not central ICD.  
\end{exmp} 

\noindent 
In the following we characterize central ICD. Let $\Real^+$ be the set of all positive real numbers.

\begin{thm}\label{cicdop}
Let $G=(V,E)$ be a simple digraph. Then $G$ is a central ICD if and only if there exist a distinct  
labeling $\Map{f}{V}{\Real^+}$ of vertices \footnote{Given a positive real number labeling, one can easily obtain a positive rational number labeling with slight adjustment and again those can be changed to positive integers by scaling as required. Thus we note that natural number labeling will produce the same class of digraphs.} which satisfies the following condition:
\begin{equation}\label{cicd1}
d(i,j)<d(i,k)\text{ for all }i,j,k\in\set{1,2,\ldots,n}\text{ such that }v_iv_j\in E\text{ and }v_iv_k\not\in E
\end{equation}
where $d(i,j)=|f(v_{i})-f(v_{j})|$ for all $i,j\in\{1,2,\hdots,n\}$.\\
(i.e., every out-neighbor distance is less than every non-out-neighbor distance from a vertex)

\end{thm}

\begin{proof}
Suppose $G=(V,E)$ be a central ICD with $V=\set{v_1,v_2,\ldots,v_n}$ and\\
 $\Set{(I_i,c_i)}{i=1,2,\ldots,n}$ be a central point-interval representation of $G$ where $v_i$ corresponds to $(I_i,c_i)$, i.e., $v_iv_j\in E$ if and only if $c_j\in I_i$ and assume that $v_i$'s are ordered according to increasing sequence of $c_i$'s (without loss of generality, we also assume that $c_i$'s are distinct). Define a labeling $\Map{f}{V}{\Real^+}$ by $f(v_i)=c_i$. This vertices are ordered according to increasing order of their labels. Now let $i,j,k\in\set{1,2,\ldots,n}$ such that $v_iv_j\in E$ and $v_iv_k\not\in E$. We consider following cases:

\vspace{1em}\noindent
{\bf Case I:}\ $i<j<k$ or, $k<j<i$. Then $d(i,k)=d(i,j)+d(j,k)>d(i,j)$ for $d(j,k)>0$ as $c_i$'s are distinct.

\vspace{1em}\noindent
{\bf Case II:}\ $i<k<j$ or, $j<k<i$. These cases are not possible by (\ref{icd1}) as $G$ is an ICD.

\vspace{1em}\noindent
{\bf Case III:}\ $j<i<k$ or, $k<i<j$. Now $v_iv_j\in E$ and $v_iv_k\not\in E$ imply $c_j\in I_i$ but $c_k\not\in I_i$. Let $r=\frac{|I_i|}{2}$. Since $c_i$ is the central point of $I_i$, we have $I_i=[c_i-r,c_i+r]$. Then $d(i,j)=|c_i-c_j|<r<|c_i-c_k|=d(i,k)$.

\vspace{1em}\noindent 
Conversely, suppose $G$ satisfies (\ref{cicd1}) with a labeling $f$. Let us arrange vertices according to increasing order of their labels. For each $i=1,2,\ldots,n$, define $c_i=f(v_i)$. Let $i_1$ and $i_2$ be the least and the highest numbers such that $i_1=i$ or, $v_iv_{i_1}\in E$ and $i_2=i$ or, $v_iv_{i_2}\in E$. Note that $i_1\leqslant i\leqslant i_2$. Define $r_i=\max\set{d(i,i_1),d(i,i_2)}$ and $I_i=[c_i-r_i,c_i+r_i]$. We show that $\Set{(I_i,c_i)}{i=1,2,\ldots,n}$ is a central point-interval representation of $G$, i.e., $G$ is a central ICD. As $c_{i}$ is the center point of $I_{i}$ for each $i\in\{1,\hdots,n\}$, it is sufficient to prove that $G$ is an ICD.

\vspace{1em}\noindent
We verify (\ref{icd1}) to show that $G$ is an ICD. Let $i<j<k$ and $v_iv_k\in E$. Now $d(i,k)=d(i,j)+d(j,k)>d(i,j)$. So $v_{i}v_{j}\in E$. Let $v_kv_i\in E$. Again $d(k,i)=d(k,j)+d(j,i)>d(k,j)$. So $v_{k}v_{j}\in E$ as required.
\end{proof}

\begin{exmp}
Consider the digraph $D_{10}$ in Figure \ref{relation1}. It can be easily checked that $v_1<v_2<v_3<v_4$ is an ICD ordering of $D_{10}$ and $D_{10}$ is a central ICD with the labeling: $f(v_1)=1$, $f(v_2)=3$, $f(v_3)=4$ and $f(v_4)=6$.
\end{exmp}

\noindent 
Interestingly a family of (undirected) graphs satisfying the condition analogous to (\ref{cicd1}) becomes a well known class of graphs, namely, proper interval graphs.

\begin{thm}
Let $G=(V,E)$ is a graph. Then $G$ is a proper interval graph if and only if there exist an ordering of vertices and a distinct labeling $f:V\rightarrow \mathbb{R}^+$ of vertices such that
\begin{equation}\label{proper23}
 d(u,v)<d(u,w)\ \text{ for all }\ u,v,w\in V \text{ such that }\ uv\in E\ \text{ but }\ uw\notin E,
 \end{equation}
where $d(u,v)=|f(u)-f(v)|$ for all $u,v\in V$.
\end{thm}

\begin{proof}
Let $G=(V,E)$ be a proper interval graph. By Theorem \ref{proper1} the reduced graph $\tilde{G}=(\tilde{V},\tilde{E})$ of $G$ is an induced subgraph of $G(n,r)=(V_{n},E^{'})$ for some $n,r\in N$ with $n>r$. Let $V_{n}=\{x_{1},x_{2},\hdots,x_{n}\}$ such that $x_{i} \leftrightarrow x_{j}$ in $G(n,r)$ if and only if $0<|i-j|\leq r$. Let $\tilde{V}=\{x_{i_{1}},x_{i_{2}},\hdots,x_{i_{m}}\}$ such that $i_{1}<i_{2}<\hdots<i_{m}$. For convenience, we write $y_{j}=x_{i_{j}}$ for $j=1,2,\hdots,m$. Define $f: V\rightarrow \Real^+$ by $f(u)=i_{j}+\dfrac{k}{z+1}$ if $u$ is a  $k^{\text{th}}$ copy among $z$ copies of $y_{j}$ (for any but a fixed permutation of them). We arrange the vertices of $V$ according to the increasing order of vertices in $\tilde{V}$ keeping copies of same vertices together.
Let $u,v,w\in V$ such that $u\leftrightarrow v$ and $u\nleftrightarrow w$.
Let $f(u)=i_{p},f(v)=i_{q}$ and $f(w)=i_{t}$ for some $p,q,t\in\{1,2,\hdots,n\}$. Then
$|i_{p}-i_{q}|\leq r$ and $|i_{p}-i_{r}|>r$. So $d(u,v)=|f(u)-f(v)|<|f(u)-f(w)|=d(u,w)$. Therefore $d(u,v)<d(u,w)$ for all $u,v,w\in V$ such that $u\leftrightarrow v$ and $u\nleftrightarrow w$. 

\vspace{.1cm}
\noindent Conversely, let $G=(V,E)$ satisfies (\ref{proper23}) with a labeling $f$. We arrange the vertices of $G$ according to the increasing order of their labels.

\noindent Let $i<j<k,$ where $i,j,k\in \{1,2,\hdots,n\}$ and $v_{i}\leftrightarrow v_{k}$. Then $d(v_{i},v_{j})=d(v_{i},v_{k})-d(v_{j},v_{k})<d(v_{i},v_{k})$.
Also $d(v_{j},v_{k})=d(v_{i},v_{k})-d(v_{i},v_{j})<d(v_{i},v_{k})$. Then by (\ref{proper23}), $v_{i}\leftrightarrow v_{j}$ and $v_{j}\leftrightarrow v_{k}$. 
Thus $G$ satisfies umbrella property and hence by Theorem \ref{proper1}, $G$ becomes a proper interval graph.
\end{proof}

\section{Oriented interval catch digraph}

\noindent We first note that the characterization of an oriented ICD is obvious. A simple digraph $G=(V,E)$ is an oriented ICD if and only if $G$ is oriented and there exists an ordering of vertices in $V$ satisfying (\ref{icd1}) or, equivalently, there exists an ordering of vertices in $V$ such that $A^*(G)=(a_{i,j})$ satisfies consecutive $1$'s property for rows and for each pair $i\neq j$, $a_{i,j}=1$ implies $a_{j,i}=0$. Now we study some important properties of oriented ICD.

\begin{thm}\label{dag}
An oriented ICD is a directed acyclic graph (DAG).
\end{thm}

\begin{proof}
Let $G$ be an oriented ICD. We first show that $G$ has no induced directed cycles. Note that any induced subdigraph of an ICD is also an ICD by definition. Now it is easy to check that no induced directed cycles of length $\geq 3$ satisfies (\ref{icd1}) and so they are not ICD. Finally, since the digraph is oriented, there are no $2$-cycles.

\vspace{1em}\noindent
Now we use induction on the length of the cycle. Since $G$ has no induced directed cycles, in particular, $G$ has no directed cycles of length $3$. Suppose there are no directed cycles of length less than $k$ and $G$ has a directed cycle $C$ of length $k>3$. Since $C$ is not induced, there must be a chord. Interestingly, every such arc (chord) forms a smaller directed cycle on one side of it along with some arcs of $C$. This contradicts the induction hypothesis.
\end{proof}

\noindent
The following corollary is immediate from the above theorem.

\begin{cor}\label{ss1}
Let $G=(V,E)$ be an oriented ICD. Then every induced $3$-cycle of $U(G)$ is transitively oriented in $G$.
\end{cor}

\noindent
In \cite{Maehera} Maehara proved that an acyclic ICD is a central ICD. Thus we have the following:

\begin{cor}
Every oriented ICD is a central ICD.
\end{cor}

\begin{defn}
Let $G$ be an oriented digraph and $C$ be an even chordless cycle of $G$ which is not a directed cycle. Then $C$ is said to be {\em alternatively oriented} if any two arcs of $C$ with common end point have opposite directions.  
\end{defn}

\begin{prop}\label{alternative}
Let $G=(V,E)$ be an oriented ICD and $C$ be a chordless $4$-cycle of $G$. Then $C$ is alternatively oriented.
\end{prop}

\begin{proof}
Let $C=(v_{1},v_{2},v_{3},v_{4},v_1)$ be a chordless $4$-cycle in $G$. By Theorem \ref{dag}, $C$ is not a directed $4$-cycle. Then there are only $3$ other (non-isomorphic) options which are described in Figure \ref{t123}. The first option $T_1$ is alternatively oriented. The two other options are not ICD. For example, in the case of $T_2$, in order to find an ICD ordering we need the elements of each of the sets $\set{v_1,v_2,v_4}$, $\set{v_2,v_3}$ and $\set{v_3,v_4}$ to be consecutive. Since the ordering is linear and we require elements of sets $\set{v_2,v_3}$ and $\set{v_3,v_4}$ to be consecutive in that ordering. So we have either $v_2,v_3,v_4$ or $v_4,v_3,v_2$ consecutive in that ordering. But then $\set{v_1,v_2,v_4}$ cannot be consecutive. So $T_2$ is not ICD. Similarly, it can be shown that $T_3$ is not an ICD.
\end{proof}  

\begin{figure}[t]
\begin{center}

\begin{tikzpicture}[scale=0.75]

\draw (0,0)--(2,0);
\draw (2,0)--(2,2);
\draw (0,0)--(0,2);
\draw (2,2)--(0,2);

\draw (4,0)--(6,0);
\draw (4,0)--(4,2);
\draw (6,0)--(6,2);
\draw (4,2)--(6,2);

\draw (8,0)--(10,0);
\draw (8,0)--(8,2);
\draw (10,0)--(10,2);
\draw (8,2)--(10,2);

\draw  [fill=black](0,0) circle [radius=0.05];
\draw  [fill=black](2,0) circle [radius=0.05];
\draw  [fill=black](2,2) circle [radius=0.05];
\draw  [fill=black](0,2) circle [radius=0.05];
\draw  [fill=black](4,0) circle [radius=0.05];
\draw  [fill=black](6,0) circle [radius=0.05];
\draw  [fill=black](4,2) circle [radius=0.05];
\draw  [fill=black](6,2) circle [radius=0.05];

\draw  [fill=black](8,0) circle [radius=0.05];
\draw  [fill=black](10,0) circle [radius=0.05];
\draw  [fill=black](8,2) circle [radius=0.05];
\draw  [fill=black](10,2) circle [radius=0.05];

\node [left] at (0,0) {$v_{4}$};
\node [left] at (0,2) {$v_{1}$};
\node [right] at (2,0) {$v_{3}$};
\node [right] at (2,2) {$v_{2}$};

\node [left] at (4,0) {$v_{4}$};
\node [left] at (4,2) {$v_{1}$};
\node [right] at (6,0) {$v_{3}$};
\node [right] at (6,2) {$v_{2}$};

\node [left] at (8,0) {$v_{4}$};
\node [left] at (8,2) {$v_{1}$};
\node [right] at (10,0) {$v_{3}$};
\node [right] at (10,2) {$v_{2}$};

\draw [->](2,0)--(1,0);
\draw [->](0,2)--(0,1);
\draw [->](2,0)--(2,1);
\draw [->](0,2)--(1,2);

\draw [->](6,0)--(5,0);
\draw [->](6,2)--(6,1);
\draw [->](4,2)--(5,2);
\draw [->](4,2)--(4,1);

\draw [->](8,0)--(9,0);
\draw [->](8,2)--(8,1);
\draw [->](8,2)--(9,2);
\draw [->](10,2)--(10,1);

\node [below] at (1,-0.5){ $T_{1}$};
\node [below] at (5,-0.5){ $T_{2}$};
\node [below] at (9,-0.5){ $T_{3}$}; 
\end{tikzpicture}

\end{center}

\vspace{-1em}
\caption{Orientation of $4$-cycles which are not directed $4$-cycles.}\label{t123}
\end{figure}

\noindent
Now being an oriented acyclic digraph, an oriented ICD has some rich properties. For example, a unilaterally connected oriented ICD has unique source (vertex of indegree zero) and unique sink (vertex of outdegree zero). Hence it possesses unique hamiltonian path (see \cite{Bang}). In the following we characterize oriented ICDs which are tournaments.

\vspace{1em} \noindent 
A {\em Ferrers digraph} is a directed graph $G=(V, E)$ whose successor sets are linearly ordered by inclusion, where the successor set of $u\in V$ is its set of out-neighbors $\Set{v\in V}{uv \in E}$. A $(0,1)$-matrix $M$ is a {\em Ferrers matrix} if $1$'s are clustered in a corner of $M$. A digraph $G$ is Ferrers digraph if and only if there exists a permutation of vertices of $G$ such that its adjacency matrix is a Ferrers matrix \cite{BDGS}. For a $(0,1)$-matrix $M$, $\overline{M}$ denotes the matrix obtained from $M$ by interchanging $0$'s and $1$'s.

\begin{thm}
Let $G=(V,E)$ be an oriented ICD. Then $G$ is a tournament if and only if there is an ordering of vertices of $G$ with respect to which the augmented adjacency matrix $A^{*}(G)$ takes one of the following forms:
$${
A^{*}(G)=
\left[
\begin{array}{c|c}
M & F \\
\hline
\overline{F^{T}} & N
\end{array}
\right]=
}
{
\begin{array}{ccccc|cccccccccccc}
 & v_1 & v_2 & . & v_{l} & v_{l+1} & . & v_k& v_n \\

\cline{2-9}
\multicolumn{1}{c@{\,\vline}}{v_1} & 1 & 1 & 1 & 1 & 1 & 1  &\multicolumn{1}{c@{\,\vline}}{1} &\multicolumn{1}{c@{\,\vline}}{0}  \\

\cline{8-8}
\multicolumn{1}{c@{\,\vline}}{v_2} & 0 & 1 & 1 & 1 & 1 & \multicolumn{1}{c@{\,\vline}}{1} & 0 &\multicolumn{1}{c@{\,\vline}}{0} \\

\cline{7-7}
\multicolumn{1}{c@{\,\vline}}{.} & 0 & 0 &1 & 1 &\multicolumn{1}{c@{\,\vline}}{1} & 0 & 0 & \multicolumn{1}{c@{\,\vline}}{0}\\

\cline{6-6}
\multicolumn{1}{c@{\,\vline}}{v_{l}} & 0 & 0 & 0 & 1 & 0 & 0 & 0 & \multicolumn{1}{c@{\,\vline}}{0} \\

\cline{2-9}

\multicolumn{1}{c@{\,\vline}}{v_{l+1}} & 0 & 0 & \multicolumn{1}{c@{\,\vline}}{0} & 1 & 1 & 0 & 0 & \multicolumn{1}{c@{\,\vline}}{0}\\

\cline{4-4}

\multicolumn{1}{c@{\,\vline}}{.} & 0 &\multicolumn{1}{c@{\,\vline}}{0} & 1 & 1 & 1 & 1 & 0 & \multicolumn{1}{c@{\,\vline}}{0}\\

\cline{3-3}

\multicolumn{1}{c@{\,\vline}}{v_k} &\multicolumn{1}{c@{\,\vline}}{0} & 1 & 1 & 1 & 1 &1 &1 & \multicolumn{1}{c@{\,\vline}}{0}\\

\cline{2-2}
\multicolumn{1}{c@{\,\vline}}{v_n} & 1 & 1 & 1 & 1 & 1 & 1 & 1 & \multicolumn{1}{c@{\,\vline}}{1}\\

\cline{2-9}
\end{array}
}\ \text{or}\ 
A^{*}(G)=
\begin{array}{cccccc}
 &v_{1}&.& . &v_{n}\\
\cline{2-5}
\multicolumn{1}{c@{\,\vline}}{v_1} & 1 & 0 & 0 & \multicolumn{1}{c@{\,\vline}}{0}\\
\multicolumn{1}{c@{\,\vline}}{.} & 1 & 1 & 0 & \multicolumn{1}{c@{\,\vline}}{0}\\
\multicolumn{1}{c@{\,\vline}}{.} & 1 & 1 & 1 & \multicolumn{1}{c@{\,\vline}}{0} \\
\multicolumn{1}{c@{\,\vline}}{v_n} & 1 & 1 & 1 & \multicolumn{1}{c@{\,\vline}}{1}\\
\cline{2-5}
\end{array}
= P$$
\noindent
where $M$ is an upper triangular matrix with all entries on and above the principal diagonal are $1$ and other enties are $0$; $N$ and $P$ are lower triangular matrices with all entries on and below the principal diagonal are $1$ and other entries are $0$ and $F$ is a Ferrers matrix.
\end{thm}

\begin{proof}
Let $G=(V,E)$ be an oriented ICD which is a tournament. We order the vertices of $V$ according to the increasing value of the associated points in their corresponding intervals. Let $\{v_{1},\hdots,v_{n}\}$ be the required ordering which is in fact, an ICD ordering, i.e., $\{v_{1},\hdots,v_{n}\}$ satisfies (\ref{icd1}). Hence $A=A^{*}(G)$ satisfies consecutive $1$'s property for rows with respect to this ordering. For each $1\leq i \leq n$, let ${(v_{i})}_{1}$ and ${(v_{i})}_{2}$ be the column numbers where first and last $1$ occur in the $i$-th row. We denote the $(i,j)$-th entry of the matrix $A$ by $a_{i,j}$.  

\vspace{0.25em}\noindent 
{\bf Case I:}\ ${(v_{1})}_{2}=k>1$.

\vspace{0.25em}\noindent 
In this case, $a_{1,2}=1$ (as $k>1$). So $a_{2,1}=0$ as $G$ is a tournament. Hence ${(v_{1})}_{1}=1<2={(v_{2})}_{1}$. Now the following subcases may happen.

\begin{itemize}
\item[{\bf Case I(a):}] If ${(v_{2})}_{2}=2$ then $2={(v_{2})}_{2}={(v_{2})}_{1}$ which implies the vertex $v_{2}$ has out degree $0$. In this case we define $l=2$ (note that the matrix $M$ in the statement is of the order $l\times l$).

\item[{\bf Case I(b):}] If ${(v_{2})}_{2}>2$, then we can show that ${(v_{2})}_{2}\leq {(v_{1})}_{2}$. If not, let ${(v_{2})}_{2}>{(v_{1})}_{2}$. Then $a_{1,(v_{2})_{2}}=0$ which implies $a_{{(v_{2})}_{2},1}=1$ as $G$ is a tournament. Now $a_{2,(v_2)_2}=1$ by definition. Then $a_{(v_{2})_{2},2}=0$. But we have $a_{{(v_{2})}_{2},1}=1$ and $a_{{(v_{2})}_{2}, {(v_{2})}_{2}}=1$. This contradicts the consecutive $1$'s property for the ${(v_{2})}_{2}$-th row as $1<2<(v_2)_2$. Hence ${(v_{2})}_{2}\leq {(v_{1})}_{2}$. 

\vspace{0.25em}\noindent
Now since ${(v_{2})}_{2}>2$, $a_{2,3}=1$ and hence $a_{3,2}=0$. Thus ${(v_{3})}_{1}=3$. Let 
$$S=\Set{i}{{(v_{i})}_{1}=i,\ {(v_{i})}_{2}\leq {(v_{i-1})}_{2},\ 2\leq i\leq n}.$$ 
Already we have shown that $2\in S$. This implies $S\neq \emptyset$. As $|V|$ is finite $S$ must have a maximum, say, $l$. First we will show that $\Set{i}{2\leq i\leq l}\subseteq S$. As $2\in S$, applying induction we assume $\set{2,3,\ldots,j}\subseteq S$. If $j+1\leq l$, then we  show that $j+1\in S$. 

\vspace{0.25em}\noindent
If ${(v_{j})}_{2}=j$ then $a_{j,l}=0$ as $l>j$. Thus $a_{l,j}=1$ and so $(v_l)_1\leq j<l$ which contradicts $l\in S$. Therefore $(v_j)_2>j$. But then $a_{j,j+1}=1$ which implies $a_{j+1,j}=0$ and so $(v_{j+1})_1=j+1$. Next we show $(v_{j+1})_2\leq (v_j)_2$.

\vspace{0.25em}\noindent
On contrary let $(v_{j+1})_2>(v_j)_2$. Since $(v_j)_2>j$, we have $(v_{j+1})_2\geq j+1$. Now if ${(v_{j+1})}_{2}=j+1$, then ${(v_{j+1})}_{2}\leq (v_j)_2$ as $(v_j)_2>j$. So suppose ${(v_{j+1})}_{2}>j+1$. Now $a_{j,(v_{j+1})_2}=0$ as we have assumed $(v_{j+1})_2>(v_j)_2$. Thus $a_{(v_{j+1})_2,j}=1$. Again $a_{j+1,(v_{j+1})_2}=1$ implies $a_{(v_{j+1})_2,j+1}=0$. But $a_{(v_{j+1})_2,(v_{j+1})_2}=1$. This contradicts the consecutive $1$'s property for $(v_{j+1})_2$-th row as $j<j+1<(v_{j+1})_2$. So we must have $(v_{j+1})_2\leq (v_j)_2$ and hence $j+1\in S$. This completes the induction. Therefore $\Set{i}{2\leq i\leq l}\subseteq S$ and so 
$${(v_{i})}_{1}=i, {(v_{i})}_{2}\leq {(v_{i-1})}_{2}\text{ for all }i\text{ with }2\leq i\leq l.$$
\end{itemize}

\noindent
Now we will show $l\leq k=(v_1)_2$ and ${(v_{l})}_{2}=l$. On contrary let $l>k$. Then $a_{1,l}=0$ as $k=(v_1)_2$. This imply $a_{l,1}=1$. Then $({v_{l})}_{1}=1<l$ which contradicts $l\in S$. Therefore $l\leq k$. 

\vspace{0.5em}\noindent
Now if ${(v_{l})}_{2}>l$ then $a_{l,l+1}=1$ which imply $a_{l+1,l}=0$, i.e., ${(v_{l+1})}_{1}=v_{l+1}$. But $l+1\notin S$. This implies ${(v_{l+1})}_{2}>{(v_{l})}_{2}>l$. Hence $a_{l,(v_{l+1})_{2}}=0$ which imply $a_{(v_{l+1})_{2},l}=1$. But $a_{{(v_{l+1})}_{2},l+1}=0$ as $a_{l+1,(v_{l+1})_2}=1$ and $a_{{(v_{l+1})}_{2},{(v_{l+1})}_{2}}=1$. Thus $(v_{l+1})_2\neq l+1$ and so $(v_{l+1})_2>l+1$. Then we have a contradiction in the consecutive $1$'s property for the ${(v_{l+1})}_{2}$-th row as $l<l+1<(v_{l+1})_2$. So ${(v_{l})}_{2}=l$. Already we have ${(v_{l})}_{1}=l$. Thus the $l$-th row contains only one $1$ at the position $(l,l)$.

\vspace{0.5em}\noindent  
Therefore the rows $\set{1,2,\ldots,l}$ form upper triangular matrix $M$ with all entries on and above diagonal as $1$ and $F$ becomes a Ferrers matrix formed by the rows $\set{1,2,\ldots,l}$ and columns $\set{l+1,l+2,\ldots,n}$ as ${(v_{l})}_{2}=l\leq{(v_{l-1})}_{2}\leq \ldots \leq {(v_{1})}_{2}=k$. As $G$ forms a tournament, matrix formed by rows $\set{l+1,l+2,\ldots,n}$ and columns $\set{1,2,\ldots,l}$ is $\overline{F^{T}}$.

\vspace{0.5em}\noindent
Now the only remaining thing is to show ${(v_{i})}_{2}=i$ for all $i>l$. On contrary, let there exist some $i$ where $l<i\leq n$ for which ${(v_{i})}_{2}>i$. From this it follows $a_{i,i+1}=1$ which implies $a_{i+1,i}=0$ i.e., ${(v_{i+1})}_{1}=i+1$. Now since $(v_l)_2=l$, $a_{l,i+1}=0$ and so $a_{i+1,l}=1$ which is a contradiction as ${(v_{i+1})}_{1}=i+1>l$. So we have ${(v_{i})}_{2}=i$ for all $i>l$. Thus we get the required lower triangular matrix $N$ formed by the vertices $\set{v_{l+1},v_{l+2},\ldots,v_{n}}$.

\vspace{0.25em}\noindent 
{\bf Case II:}\ ${(v_{1})}_{2}=v_{1}$. 

\vspace{0.25em}\noindent
In this case $a_{1,j}=0$ for all $j$ such that $1< j\leq n$. Since $G$ is a tournament this implies $a_{j,1}=1$ for all such $j$. Then all the entries in the first column of $A$ are $1$. Also since $A$ is the augmented adjacency matrix of $G$, $a_{j,j}=1$ for all $j=1,2,\ldots,n$. Moreover, $A$ satisfies consecutive $1$'s property for rows. Thus we have $a_{j,i}=1$ for all $i$ such that $1\leq i\leq j$, for all $j=1,2,\ldots,n$. Finally since $G$ is a tournament $a_{i,j}=0$ for all $i<j$, for all $j=1,2,\ldots,n$. Therefore $A$ is the lower triangular matrix $P$ with all entries $1$ on and below the principal diagonal. 

\vspace{0.25em}\noindent 
The converse part is obvious from the structure of $A^{*}(G)$ as both the matrices described in the statement have consecutive $1$'s property for rows and for each pair $i\neq j$, $a_{i,j}=0$ if and only if $a_{j,i}=1$. Thus $G$ is an ICD which is a tournament.
\end{proof}

\noindent Next we characterize an oriented ICD $G$ for which its underlying undirected graph $U(G)$ is a tree. We characterize the aforesaid digraphs in terms of forbidden subdigraphs.

\begin{thm}\label{d2}
Let $G=(V,E)$ be a oriented digraph such that $U(G)$ is a tree. Then $G$ is an oriented ICD if and only it does not contain any of $D_{1}, D_{2},K$ as an induced subdigraph.
\end{thm}

\begin{proof}
Let $G$ be a oriented ICD. As $D_{1}$ contains $v,v_{4},v_{5}$, $D_{2}$ contains $v,v_{2},v_{3}$ and $K$ contains $x,y,z$ as diasteroidal triple,
from Theorem \ref{diate} it follows that $G$ can not contain $D_{1},D_{2},K$ as induced subdigraph.

\noindent Conversely let $G$ be a oriented digraph which does not contain $D_{1}, D_{2},K$ as induced subdigraph. Let $S$ be the set of all chains connecting every pair of vertices in $G$ and $\vec{P}=(x=a_{0},a_{1},\hdots, a_{l-1},a_{l}=y)$ be a chain of maximum length (say $l$) between two vertices $x,y$ (say) of $V$.
As $U(G)$ does not contain $K$ as induced subdigraph, it becomes a caterpillar and $U(\vec{P})$ becomes the spine of the caterpillar. Let $a_{i}, i\neq 1,l-1$ be an internal vertex of $V(\vec{P})$. Then looking at $D_{1}$ and $D_{2}$ one can verify that $a_{i}$ can not possess any inneighbour from $V\setminus V(\vec{P})$.
 Only $a_{1},a_{n-1}$ can have inneighbour from $V\setminus V(\vec{P})$ and that must also be one in number as $D_{2}$ is forbidden in $G$. Let $k_{1}(k_{2})$ be the only  possible inneighbour of $a_{1}(a_{n-1})$ outside $V(\vec{P})$ in $G$.
Now if any of the vertices $k_{1},k_{2}$ presents in $G$ then we 
construct the modified chain $\vec{\hat{P}}=(b_{0},b_{1},\hdots, b_{l-1},b_{l})$ by only replacing the vertices $a_{0},a_{l}$ of $V(\vec{P})$ by $k_{1},k_{2}$ respectively. Note that $\vec{\hat{P}}$ is of same length $l$ as $\vec{P}$. 
We now attach a variable $p_{v_{i}}$ to each vertex $v_{i}$ of $G$. Then starting from $p_{b_{0}}$ we place the points $p_{b_{i}}$ corresponding to the vertices $b_{i}$ of the chain $\vec{\hat{P}}$ in increasing order in a single sequence $P_{S}$ (say) along the real line.  From construction it is easy to observe that in the chain $\vec{\hat{P}}$ there is no internal vertex containing inneighbours from $V\setminus V(\vec{\hat{P}})$.  At $i$ th step if $b_{i}$ possesses outneighbours $b_{{i}^{1}}, \hdots, b_{i^{m}}$ then we place $p_{b_{{i}^{1}}} \hdots, p_{b_{i^{m}}}$ between $p_{b_{i}},p_{b_{i+1}}$ in any order in $P_{S}$. Now we attach numbers starting from $1$ to $n$ with each point of $P_{S}$ where $|V|=n$. Let $i_{1}(i_{2})$ be the leftmost (rightmost) indices of the points, which are outneighbours of $v_{i}$ in the sequence $P_{S}$ (consider $i_{1}\leq i \leq i_{2}$) and $l_{i},r_{i},t_{i}$ be the numbers attached to $p_{i_{1}},p_{i_{2}},p_{i}$ respectively. We assign the interval $I_{i}=[l_{i},r_{i}]$ and the corresponding point $p_{i}=t_{i}$ as the value attached to $p_{v_{i}}$ in $P_{S}$ for each vertex $v_{i}\in V$. As outneighbours of each vertex occur consecutively right or left to $v_{i}$ in $P_{S}$, one can easily check from Theorem \ref{e1} that the above representation actually gives us the oriented ICD representation of $G$.
\end{proof}

\begin{figure} 
\centering
\begin{tikzpicture}
\draw[-][draw=black, thick] (0,0) -- (0,1);
\draw[-][draw=black, thick] (0,0) -- (1,-0.5);
\draw[-][draw=black, thick] (1,-0.5) -- (2,-1);
\draw[-][draw=black, thick] (0,0) -- (0,1);
\draw[-][draw=black, thick] (0,0) -- (-1,-0.5);
\draw[-][draw=black, thick] (-1,-0.5) -- (-2,-1);

\draw [fill=black] (0,0) circle [radius=0.09];
\draw [fill=black] (0,1) circle [radius=0.09];
\draw [fill=black] (1,-0.5) circle [radius=0.09];
\draw [fill=black] (2,-1) circle [radius=0.09];
\draw [fill=black] (-1,-0.5) circle [radius=0.09];
\draw [fill=black] (-2,-1) circle [radius=0.09];
\draw[->][draw=black, very thick] (0,1) -- (0,.5);
\draw[->][draw=black, very thick] (0,0) -- (-.5,-.25);
\draw[->][draw=black, very thick] (0,0) -- (.5,-.25);

\node [below] at (0,0) {{$v_{1}$}};
\node [below] at (1,-0.5) {{$v_{3}$}};
\node [below] at (2,-1) {{$v_{4}$}};
\node [below] at (-1,-0.5) {{$v_{2}$}};
\node [below] at (-2,-1) {{$v_{5}$}};
\node [above] at (0,1) {{$v$}};


\draw[-][draw=black, very thick] (4,0) -- (5,1);
\draw[-][draw=black, very thick] (4,0) -- (3,1);
\draw[-][draw=black, very thick] (4,0) -- (4,-1);

\node [left] at (4,0) {{$v_{1}$}};
\node [above] at (5,1) {{$v_{3}$}};
\node [above] at (3,1) {{$v_{2}$}};
\node [below] at (4,-1) {{$v$}};
\draw[->][draw=black, very thick] (3,1) -- (3.5,0.5);
\draw[->][draw=black, very thick] (5,1) -- (4.5,0.5);

\draw [fill=black] (4,0) circle [radius=0.05];
\draw [fill=black] (5,1) circle [radius=0.05];
\draw [fill=black] (3,1) circle [radius=0.05];
\draw [fill=black] (4,-1) circle [radius=0.05];

\node [left] at (8,1.5) {{$x$}};
\node [left] at (6,-0.5) {{$y$}};
\node [right] at (10,-0.5) {{$z$}};

\draw[-][draw=black, thick] (8,1.5) -- (8,1);
\draw[-][draw=black, thick] (8,1) -- (8,0.5);
\draw[-][draw=black, thick] (8,0.5) -- (9,0);
\draw[-][draw=black, thick] (8,0.5) -- (7,0);
\draw[-][draw=black, thick] (9,0) -- (10,-0.5);
\draw[-][draw=black, thick] (7,0) -- (6,-0.5);

\draw [fill=black] (8,1.5) circle [radius=0.09];
\draw [fill=black] (8,1) circle [radius=0.09];
\draw [fill=black] (8,0.5) circle [radius=0.09];
\draw [fill=black] (9,0) circle [radius=0.09];
\draw [fill=black] (7,0) circle [radius=0.09];
\draw [fill=black] (10,-0.5) circle [radius=0.09];
\draw [fill=black] (6,-0.5) circle [radius=0.09];
\end{tikzpicture}
\caption{$D_{1}, D_{2}, K$ (left to right) (undirected edges can have any direction)}\label{forbid5}
\end{figure}

\section{Proper interval catch digraph}

\begin{defn}
An interval catch digraph $G=(V,E)$ having a representation $\{(I_{v},p_{v})|v\in V\}$ where no interval contains other properly is called {\em proper ICD}\index{proper ICD}.
\end{defn}

\noindent Let $G=(V,E)$ be a simple digraph. Then the augmented adjacency matrix $A^{*}(G)$ has a {\em monotone consecutive arrangement (MCA)} \cite{BDGS} if and only if it has independent row and column permutations such that $1$'s appear consecutively in each row and $1_1\leq 2_1  \leq \hdots \leq n_1$ and $1_2\leq 2_2\leq \hdots\leq n_2$ where $|V|=n$ and the values ${i_{1}}$ and $i_{2}$ denote the initial column and final column containing $1$'s in the $i$-th row. Now one can seperate the $1$'s and $0$'s by drawing upper and lower stairs (polygonal path from top left to bottom right) as in Figure \ref{pic1} (right).

\vspace{0.3em}
\noindent In the following we give an adjacency matrix characterization of a proper ICD.

\begin{thm}\label{picd1}
Let $G=(V,E)$ be a simple digraph. Then $G$ is a proper ICD if and only if there exists a vertex ordering $\set{v_1,v_2,\ldots,v_n}$ of $V$ with respect to which the augmented adjacency matrix $A^{*}(G)$ satisfies the following conditions:

\begin{enumerate}	
\item[(1)] $A^{*}(G)$ satisfies consecutive $1$'s property for rows.
\item[(2)] For $i\neq j$, if $i_{1}<j_{1}$ then $i_{2}\leq j_{2}$ where $i_{1}$ and $i_{2}$ be the first and last column numbers containing $1$ in the $i$-th row where $1\leq i\leq n$ in $A^*(G)$.
\end{enumerate}	
\end{thm}

\begin{figure}[t]

{\footnotesize
\[\begin{array}{c|ccccc|}
\multicolumn{1}{c}{} & v_1 & \multicolumn{1}{c}{v_2} & \multicolumn {1}{c}{v_3} & \multicolumn{1}{c}{v_4} &\multicolumn{1}{c}{v_{5}}  \\
\cline{2-6}
 v_1 & 1 & 1 & 0 & 0 & 0\\
v_2 & 0 & 1 & 1 & 1 & 0 \\
v_3 & 1 & 1 & 1 & 0 & 0\\
v_4 & 0 & 0 & 0 & 1 & 1\\
v_5 & 0 & 0 & 1 & 1 & 1 \\

\cline{2-6}
\end{array}
\begin{array}{c|ccccc|}
\multicolumn{1}{c}{} & v_1 & \multicolumn{1}{c}{v_2} & \multicolumn {1}{c}{v_3} & \multicolumn{1}{c}{v_4} &\multicolumn{1}{c}{v_{5}}  \\
\cline{2-6}
 v_1 & 1 & 1 & 0 & 0 & 0\\
v_3 & 1 & 1 & 1 & 0 & 0 \\
v_2 & 0 & 1 & 1 & 1 & 0\\
v_5 & 0 & 0 & 1 & 1 & 1\\
v_4 & 0 & 0 & 0 & 1 & 1 \\

\cline{2-6}
\end{array}
\begin{array}{c|c|c|c|c|c|ccc}
\multicolumn{1}{c}{} & \multicolumn{1}{c}{p_1} & \multicolumn{1}{c}{p_{2}} & \multicolumn{1}{c}{p_3} & \multicolumn{1}{c}{p_4} &\multicolumn{1}{c}{p_{5}}& \multicolumn{1}{c}{} && \\
\multicolumn{1}{c}{}&\multicolumn{1}{c}{1.5}&\multicolumn{1}{c}{2}&\multicolumn{1}{c}{4}&\multicolumn{1}{c}{6}&\multicolumn{1}{c}{8}&&&\\
\multicolumn{1}{c}{} & \multicolumn{1}{c}{v_1} & \multicolumn{1}{c}{v_{2}} & \multicolumn{1}{c}{v_3} & \multicolumn{1}{c}{v_4} &\multicolumn{1}{c}{v_{5}}& \multicolumn{1}{c}{}&&  \\
\cline{2-6}
v_1 & \multicolumn{1}{c}{1} &2 & \multicolumn{1}{c}{3}  & \multicolumn{1}{c}{} &&&[1,3]& I_{1} \\ 
\cline{4-4}
v_3 & \multicolumn{1}{c}{}  &\multicolumn{1}{c}{}  &  4&\multicolumn{1}{c}{5} &&&[1.5,5]& I_{3}  \\
\cline{5-5}
\cline{2-2} 
v_2 &   & \multicolumn{1}{c}{} & \multicolumn{1}{c}{}&6 & 7 &&[2,7]&I_{2}\\
\cline{6-6}
\cline{3-3}
v_5 & \multicolumn{1}{c}{} &  & \multicolumn{1}{c}{} & \multicolumn{1}{c}{}&8&9 & [4,9] & I_{5} \\
\cline{4-4}
v_4  & \multicolumn{1}{c}{} &  \multicolumn{1}{c}{} & & \multicolumn{1}{c}{}&&10& [6,10] & I_{4} \\
\cline{2-6}

\end{array}\]
\caption{Augmented adjacency matrix $A=A^*(G)$ of a proper ICD $G$ (left), the row permuted matrix $B$ obtained from $A$ that satisfies MCA property (middle) and a proper ICD representation of $G$ (right).}\label{pic1}
}
\end{figure}

\begin{proof}	
Let $G=(V,E)$ be a proper ICD with representation $\Set{(I_{i},p_{i})}{i=1,2,\ldots,n}$ where $V=\set{v_1,v_2,\ldots,v_n}$, $I_i=[a_i,b_i]$ and $p_i$'s are distinct for $i=1,2,\ldots,n$. Then arranging the vertices according to increasing order of $p_{i}$'s we have consecutive $1$'s property for rows in $A^{*}(G)=(a_{i,j})$. We note that for $i\neq j$, $a_{i,j}=1$ if and only if $v_iv_j\in E$ if and only if $p_j\in I_i$ if and only if $i_1\leq j\leq i_2$ by definition. 

\vspace{0.3em}\noindent
Now for $i\neq j$ suppose $i_{1}<j_{1}$ and $i_{2}>j_{2}$ hold simultaneously in $A^{*}(G)$. Then $p_{i_{1}}\notin I_{j}$ as $i_1<j_1$. Again $i_1<j_1\leq j$ implies $p_{i_1}<p_{j_1}\leq p_j$. Since $p_j\in I_j$, the interval $I_j$ lies entirely right to $p_{i_1}$. So we get $a_{i}<p_{i_{1}}<a_{j}$. As the intervals do not contain other properly we have $b_{i}<b_{j}$. Now as $p_{i_{2}}\in I_{i}$ and $p_{i_{2}}>p_{j}$ (for $i_2>j_2\geq j$) we get $a_{j}\leq p_{j}<p_{i_{2}}<b_{i}<b_{j}$. Then $p_{i_{2}}\in I_{j}$ which is a contraction as $i_2>j_2$. Hence $i_2\leq j_2$.\footnote{Note that, it follows from the condition (2) that ``for $i\neq j$, $i_2<j_2$ implies $i_1\leq j_1$''. For otherwise, let $i_2<j_2$ and $i_1>j_1$. Then by the condition (2), we have $j_2\leq i_2$ (as $j_1<i_1$) which is a contradiction.}
	
\vspace{0.5em}\noindent 
Conversely, let $G$ satisfy conditions $(1)$ and $(2)$. We permute the rows of $A=A^{*}(G)=(a_{i,j})$ according to the non-decreasing order of $i_{1}$'s keeping the columns intact (see Figure \ref{pic1}). If there exist pair of rows $i,j$ for which $i_{1}=j_{1}$, then place $i$-th row prior to $j$-th row when $i_{2}<j_{2}$. If $i_{1}=j_{1}$ and $i_{2}=j_{2}$ then keep $i$-th row prior to $j$-th row only when $i<j$ in $A$.
Let us call the modified matrix by $B=(b_{i,j})$. As $A$ satisfies $(2)$, after permuting only the rows of $A$ it is easy to check that $B$ satisfies MCA. Moreover permutation of rows of $A$ does not affect the adjacency of the graph $G$, i.e., if $i$-th row of $A$ is shifted to $k$-th row in $B$, then 
\begin{equation}\label{aij1}
a_{i,j}=1\text{ if and only if }b_{k,j}=1.
\end{equation} 

\vspace{0.5em}\noindent
{\bf Step I:}\ We show that the diagonal entries of the matrix $B$ must be $1$. \\
Suppose the $i$-th row of $B$ was the $k$-th row of $A$. If $k=i$ then $b_{i,i}=1$ as $a_{i,i}=1$. Suppose $k>i$ in $A$. For each $j=1,2,\ldots, i$, $j_1\leq j\leq i$. Now the $k$-th row of $A$ moved upward to the $i$-th row position in $B$. This implies there exists at least one such $j$ so that $k_1\leq j_1$. Then $k_{1}\leq i<k\leq k_{2}$. Hence $a_{k,i}=1$ which implies $b_{i,i}=1$.  Again if $k<i$ in $A$ then there exists at least one $j\geq i$ such that $j_1<k_1$ or, $j_1=k_1$ and $j_2<k_2$. Then by the condition (2), in either case, we have $j_2\leq k_2$. Thus $k_1\leq k<i\leq j\leq j_2\leq k_2$. This implies $a_{k,i}=1$ and hence $b_{i,i}=1$.

\vspace{0.5em}\noindent 
{\bf Step II:}\ Defining intervals $I_i=[l_i,r_i]$ for the $i$-th row of $B$. \\
Now we associate natural numbers in increasing order on the upper stair of $B$ starting from the left top of it. Let $b_{i}$ be the number on the stair in the $i$-th row where the stretch of consecutive $1$'s ends and $a_i$ be the number on the stair in the $i$-th column where the stretch of consecutive $1$'s begins (e.g., in Figure \ref{pic1}, $a_i=1,2,4,6,8$ and $b_i=3,5,7,9,10$ for $i=1,2,3,4,5$ respectively). Let $i_o$ and $i_e$ denote the first and last column numbers containing $1$ in the $i$-th row of $B$ (similar to the construction of $i_1$ and $i_2$ in $A$). Then by definition of $a_i$ and $b_i$ we have $b_i>a_j$ for all $i_o\leq j\leq i_e$. \footnote{For example, in Figure \ref{pic1}, for $i=2$, $a_{2_o}=1$, $a_{2_e}=4$ and $b_2=5$; for $i=3$, $a_{3_o}=2$, $a_{3_e}=6$ and $b_3=7$; etc.} Thus 
\begin{equation}\label{bij1}
b_{i,j}=1\Longrightarrow b_i>a_j.
\end{equation}
Let $k_i=\max\Set{j}{j_o=i_o}$. We assign interval $I_{i}=[l_{i},r_{i}]$ where 
$$l_{i}=a_{i_o}+\dfrac{i-i_o}{k_i-i_o+1}\text{ and }r_{i}=b_{i}.$$
By (\ref{bij1}), $b_i>a_{i_o}$ and so $b_i\geq a_{i_o}+1>l_i$. Thus $l_i<r_i$ for all $i=1,2,\ldots,n$. Also since $B$ satisfies MCA, one can check that $l_{i}<l_{j}$ and $r_{i}<r_{j}$ for all $i<j$. Therefore the set of intervals $[l_i,r_i]$ gives proper representation. 

\vspace{0.5em}\noindent
{\bf Step III:}\ Defining points $p_{j}$ for each column $j$ of $B$ (as well as of $A$).\\
Let $S_{j}=\Set{l_{i}}{i\geq j,\ b_{i,j}=1}$. Note that $S_{j}\neq \emptyset$ as $b_{j,j}=1$. Assign $p_{j}=\max\set{a_{j}, \max S_{j}}$.

\vspace{0.5em}\noindent 
Now $p_j\geq l_j$ as $l_j\in S_j$. Suppose for some $i\geq j$, $b_{i,j}=1$. Then $i_o\leq j\leq i_e$ which implies $a_{i_o}\leq a_j$. So $l_i< a_{i_o}+1\leq a_j+1\leq b_j$. Also since $b_{j,j}=1$ by Step I, we have $a_j<b_j$ by (\ref{bij1}). Thus $p_j<b_j$. Hence 
\begin{equation}\label{pij1}
p_j\in [l_j,r_j].
\end{equation}
{\bf Step IV:}\ We show that $b_{i,j}=1$ if and only if $p_j\in I_i$.\\
{\it Case (i):}\ $b_{i,j}=1$ and $i>j$.\\
By definition of $S_j$, $l_i\in S_j$ and so $p_j\geq l_i$. Again by (\ref{pij1}), $p_j\leq r_j<r_i$ (as $i>j$). Thus $p_j\in [l_i,r_i]=I_i$.

\vspace{0.5em}\noindent
{\it Case (ii):}\ $b_{i,j}=1$ and $i<j$.\\
As $i<j$, $l_i<l_j\leq p_j$. Also since $b_{i,j}=1$, by (\ref{bij1}), $r_i=b_i>a_j$ which implies $r_i\geq a_j+1$. Moreover, for all $t>j$ with $b_{t,j}=1$, $l_t<a_j+1$ as in Step III. Thus $p_j<a_j+1\leq r_i$. Hence $p_j\in [l_i,r_i]=I_i$.

\vspace{0.5em}\noindent
{\it Case (iii):}\ $b_{i,j}=0$ and $i>j$.\\
We first show that $a_j\leq p_j<a_j+1$. By definition of $p_j$, if $p_j=a_j$, then the claim is true. If $p_j=\max S_j$, then again by definition $p_j\geq a_j$. In this case, $p_j=l_x$ for some $x\geq j$ such that $b_{x,j}=1$. Now $b_{x,j}=1$ implies $x_o\leq j$ and so $a_{x_o}\leq a_j$. By definition $a_{x_o}\leq l_x<a_{x_o}+1$. Then $p_j=l_x<a_{x_o}+1\leq a_j+1$ as required.

\vspace{0.5em}\noindent
In this case, the $\set{i,j}$-th position is left to (or, below) the lower stair, i.e., $j<i_o$. So $a_j<a_{i_o}$ which implies $a_j+1\leq a_{i_o}$. Then by the above argument, $p_j<a_j+1\leq a_{i_o}\leq l_i$ (by definition of $l_i$). Thus We have $p_j<l_i$. Hence $p_j\notin [l_i,r_i]=I_i$.

\vspace{0.5em}\noindent
{\it Case (iv):}\ $b_{i,j}=0$ and $i<j$.\\
the position $\set{i,j}$ is right to (or, above) the upper stair. Hence $b_i<a_j$. But then $r_i=b_i<a_j\leq p_j$. So $p_j$ lies right to the interval $[l_i,r_i]=I_i$. Thus $p_j\notin I_i$.

\vspace{0.5em}\noindent
{\bf Step V:}\ Finally we show that $p_j$ belongs to the interval corresponding to the vertex $v_j$.\\
Note that, if the $j$-th row of $A$ is shifted to the $k$-th row of $B$, then the vertex $v_j$ corresponds to the interval $I_k$. Now by (\ref{aij1}), we have $b_{k,j}=1$ as $a_{j,j}=1$. Then $p_j\in I_k$ by Cases (i) and (ii) of Step IV.

\vspace{0.5em}\noindent
This completes all the verifications. Therefore $G$ is a proper ICD with respect to the above representation.    
\end{proof}

\noindent An {\em unit interval catch digraph} (in brief, unit ICD) is an ICD where every interval has the same length. We note the following result without proof as it goes with the same line as the proof of $(1)\Longleftrightarrow (2)$ in Theorem \ref{proper1}.

\begin{prop}
Let $G$ be an ICD. Then $G$ is a proper ICD if and only if it is a unit ICD.
\end{prop}

\noindent Finally we define a digraph as {\em proper oriented interval catch digraph} (in brief, proper oriented ICD) if it is an oriented ICD where no two intervals are contained in other properly. Hence from Theorem \ref{e1} and Theorem \ref{picd1} we get the following:

\begin{thm}\label{poicd1}
Let $G=(V,E)$ be an oriented digraph. Then $G$ is a proper oriented ICD if and only if there exists a vertex ordering which satisfy the following
\begin{equation} \label{poicd}
\text{for} \ u<v<w, \text{if} \ uw\in E \ \text{then} \ uv,vw\in E \ \text{and if} \ wu\in E \ \text{then} \ wv,vu\in E.
\end{equation}
\end{thm}

\begin{proof}
Let $G=(V,E)$ be a proper oriented ICD. Then by Theorem \ref{picd1}, there exists an ordering of $V$ that satisfies conditions (1) and (2). Let $u<v<w$ be three vertices of $V$ in this ordering and $uw\in E$. Suppose $u,v,w$ correspond to rows $i,j,k$ respectively in $A^*(G)=(a_{i,j})$, where $i<j<k$. Since the matrix is augmented, $a_{i,i}=1$. Also $uw\in E$ implies $a_{i,k}=1$. Thus by consecutive $1$'s property for rows, we have $a_{i,j}=1$ which implies $uv\in E$. Now the digraph is oriented. Thus $a_{j,i}=a_{k,i}=0$. Then $j_1>i\geq i_1$. Thus by condition (2), $j_2\geq i_2$. But $a_{i,k}=1$ and so $i_2\geq k$. Then $j_2\geq k$ which implies $a_{j,k}=1$. Hence $vw\in E$. The other part can be proved similarly.

\vspace{0.5em}\noindent
Conversely, let $G=(V,E)$ be an oriented digraph satisfying the given condition. By Theorem \ref{e1}, it follows that $G$ is an ICD and so it satisfies the condition (1) of Theorem \ref{picd1}. Suppose in $A^*(G)=(a_{i,j})$, $i_1<j_1$ for two distinct rows $i,j$. Now if we can show that $a_{j,i_2}=1$, then it follows that $i_2\leq j_2$ as required. If $i_2\leq j$, then $i_2\leq j_2$ as $j\leq j_2$. Let $i_2>j$. If $i<j$, then $i<j<i_2$ and $a_{i,i_2}=1$. Then by the given condition (first part), we have $a_{j,i_2}=1$. Now if $i>j$, then $i_1<j_1\leq j<i$ and $a_{i,i_1}=1$. Thus by the given condition (second part), we have $a_{j,i_1}=1$. But this implies $j_1\leq i_1$ which is a contradiction.
\end{proof}

\begin{rem}
It follows from the proof of the converse part of the above theorem that in the case of a proper oriented ICD $G$, $A^*(G)$ itself satisfies MCA which may not be true for a proper ICD, in general.
We call the ordering mentioned in (\ref{poicd}) as {\em proper ICD ordering} of $V$. 
\end{rem}

\noindent Next we will give a forbidden subdigraph characterization to all those oriented interval catch digraphs which are proper and their corresponding underlying graphs are chordal.  

\begin{thm}
Let $G=(V,E)$ be an oriented ICD such that $U(G)$ is chordal. Then $G$ is a proper oriented ICD if and only if it does not contain the digraphs of Figure \ref{forbid} as induced subdigraph.
\end{thm}

\begin{proof}
Let $G$ be a proper oriented ICD with representation $\{(I_{v}=[a_{v},b_{v}],p_{v})|v\in V\}$ such that $U(G)$ is chordal. Then from Theorem \ref{poicd1} there exist a vertex ordering (say $\prec$) of $V$ satisfying (\ref{poicd}). Lets assume on contrary that $G$ contains each of the digraphs of Figure \ref{forbid} as induced subdigraph. Without loss of generality we assume $v_{2}\prec v_{4}$. As $v_{2},v_{4}$ are nonadjacent in $D_{3}, D_{6}, D_{7}, D_{8}, D_{9}$ it is easy to check that $v_{1},v_{3}$ must occur within $v_{2},v_{4}$ in the above ordering, i.e; $v_{2}\prec v_{1},v_{3}\prec v_{4}$. Let $v_{1}\prec v_{3}$. From Figure \ref{forbid} one can observe that $\{v_{2},v_{1},v_{3}\}$, $\{v_{1},v_{3},v_{4}\}$ form tournaments in the above digraphs. Hence if $v_{1}v_{3}\in E$ then $v_{2}v_{3},v_{2}v_{1},v_{1}v_{4},v_{3}v_{4}\in E$ as $G$ is oriented and $\prec$ satisfies (\ref{poicd}).
Similarly if $v_{3}v_{1}\in E$ then $v_{3}v_{2},v_{1}v_{2},v_{4}v_{1},v_{4}v_{3}\in E$. Both of these cases imply existence of same proper oriented ICD $\hat{D}$ in $G$ (see Figure \ref{allowed}), which is not isomorphic to any of the digraphs $D_{3}, D_{6}, D_{7}, D_{8}, D_{9}$. If $v_{3}\prec v_{1}$ we get same contradiction. 

\noindent Next for $D_{4}$ one can easily check that either $v_{3}\prec v_{2}\prec v_{4}$ or $v_{4}\prec v_{2}\prec v_{3}$. In both of these cases it is not possible to place $v_{1}$ anywhere in the above ordering as $v_{1}$ is adjacent (nonadjacent) to $v_{2}$ ($v_{3},v_{4}$). Similarly in $D_{5}$, $v_{1}\prec v_{2}\prec v_{3}$ or $v_{3}\prec v_{2}\prec v_{1}$ as $v_{1},v_{3}$ are nonadjacent in $U(D_{5})$. In these cases we can not place $v_{4}$ anywhere in the above ordering follows from its adjacency. Hence contradiction arises in both of these cases.

\vspace{0.3em}

\noindent Conversely, let $G$ does not contain any of the digraphs of Figure \ref{forbid} as induced subdigraph. Now as $G$ is an ICD, it must possess an ICD-ordering $P=(v_{1},v_{2},\hdots,v_{n})$ (say). We want to create a proper ICD ordering $\hat{P}$ from $P$ so that it satisfy (\ref{poicd}) and $G$ become a proper oriented ICD from Theorem \ref{poicd1}. It is known from Theorem \ref{dag} that $G$ is acyclic and it can not contain $D_{2}$ as induced subdigraph follows from Theorem \ref{d2}.
We denote the set of all outneighbours (inneighbours) of a vertex $v_{i}\in V$ by $N_{O}[v_{i}](N_{I}[v_{i}])$ (including $v_{i}$). From definition of ICD-ordering it is clear that vertices of $N_{O}[v_{i}]$ must occur consecutive along $P$. Below we list some lemmas which we repeatedly use throughout the proof.

\begin{lem}\label{lem1}
Let $v\in V$ and $v_{l},v_{r}$ be two out(in)-neighbours and $v_{k}$ be an in(out)-neighbour of $v$ in $G$. Then $v_{l}$ and $v_{r}$ must share an edge.
\end{lem}

\begin{proof}
If $v_{l},v_{r}$ do not have an edge, then the vertices $\{v_{l},v,v_{k},v_{r}\}$ induce $D_{5}$ or $D_{4}$ or $D_{6}$ in $G$, which can not happen. In other case we also get contradiction as $D_{4}, D_{8}$ are forbidden.
\end{proof}

\begin{lem}\label{lem2}
Let $v\in V$ and $v_{l},v_{r}$ be two nonadjacent out(in)-neighbours of $v$ in $G$. Then there exist no common neighbour of $v,v_{l},v_{r}$ in $U(G)$.
\end{lem}

\begin{proof}
On contrary let $v_{k}$ be the common neighbour of $v,v_{l},v_{r}$ in $U(G)$. We first consider the case when $v_{k}$ is an inneighbour of $v$. Now if $vv_{l},vv_{r}\in E$ then from Lemma \ref{lem1} there must exist an edge between $v_{l},v_{r}$ in $U(G)$, which contradicts our assumption. Again if $v_{l}v,v_{r}v\in E$ then $D_{3}$ or $D_{7}$ or $D_{8}$ get induced in $G$, which is not true.

\noindent Next we consider the case when $v_{k}$ is an out-neighbour of $v$. Now if $vv_{l},vv_{r}\in E$ then $D_{6}$ or $ D_{7}$ or $D_{9}$ gets formed by the vertices $\{v,v_{l},v_{k},v_{r}\}$. Again if $v_{l}v,v_{r}v\in E$ then from Lemma \ref{lem1} it follows that $v_{l},v_{r}$ become adjacent in $U(G)$. Hence in both of these cases contradiction arises. Thus the result holds.
\end{proof}

\begin{cor}\label{corl1}
Let $v\in V$. Then $N_{O}[v](N_{I}[v])$ is either a tournament or is disjoint union of two tournaments in $G$.
\end{cor}

\begin{proof}
Note that $v$ can not possesses three nonadjacent out(in)-neighbours in $G$ as $D_{5}$ ($D_{2}$) get induced in that case, which can not be true.
Rest of the proof follows from Lemma \ref{lem2}.
\end{proof}

\begin{defn}\label{outin}
A set $A$ of vertices is said to form an {\em Out-Fountain} if there exists an ordering of vertices such that for each $i< j,$ $v_{i}v_{j}\in E$. Similarly
a set $B$ of vertices is said to form an {\em In-Fountain} if there exists an ordering of vertices such that for each $i> j,$ $v_{i}v_{j}\in E$. A set is said to be a {\em Fountain}\index{Fountain} if and only if it is either an {\em Out-Fountain} or an {\em In-Fountain}. Further we denote the first element by $v_{f}$ and the last by $v_{l}$ of any Fountain. Note that if we reverse the order of an Out-Fountain, it becomes an In-Fountain and vice versa. Moreover any Fountain must be a tournament. 
\end{defn}
\vspace{0.1em}

\noindent\textbf{Notation:} By $A$ we denote an Out-Fountain and $B$ an In-Fountain.  By $N_{O}(A\setminus \{v_{f}\})$ $(N_{I}(B\setminus\{v_{f}\}))$ we mean the set of all out(in)-neighbours of $A\setminus \{v_f\}$ ($(B\setminus \{v_{f}\}$) which lies outside $A$ ($B$).

\begin{lem}\label{lme3}
If $v_{k}$ be an out(in)-neighbour of some element (say $v_{i}$) of $A\setminus\{v_{f}\}$ $(B\setminus\{v_{f}\})$ then $v_{k}$ is also an out(in)-neighbour of 
$v_{j}$ for all $i\leq j\leq l$.
\end{lem}

\begin{proof}
We assume on contrary that $v_{j}$, $i<j\leq l$ be a vertex which does not have $v_{k}$ as its out(in)-neighbour. Then the set $\{v_{f},v_{i},v_{j},v_{k}\}$ induces $D_{3}$ or $D_{4}$ ($D_{4}$ or $D_{9}$), which can not happen.
\end{proof}

\begin{lem}\label{lme4}
The vertices of $N_{O}(A\setminus \{v_{f}\}$ $(N_{I}(B\setminus \{v_{f}\})$ form a tournament.
\end{lem}

\begin{proof}
 Let $v_{k},v_{k'}\in N_{O}(A\setminus \{v_{f}\})$ be two nonadjacent out-neighbours of $v_{i},v_{j}$ respectively where $i<j$ and $v_{i},v_{j}\in A$. Then from Lemma \ref{lme3} it follows that both of them are out-neighbours of $v_{j}.$ Then $\{v_{f},v_{j},v_{k},v_{k'}\}$ induce $D_{5}.$ The proof for $N_{I}(B\setminus \{v_{f}\})$ follows similarly using the fact that $D_{2}$ is not an ICD.
\end{proof}

\begin{lem}\label{lme5}
Let $v_{k}$ be not an out(in)-neighbour of $v_{f}.$ Then it cannot be both an in-neighbour and an out neighbour of $A$($B$).
\end{lem}

\begin{proof}
Let $v_{i},v_{j}\in A$ where $i<j$. As oriented ICD is acyclic, $v_{j}v_{k},v_{k}v_{i}\notin E$. Hence we assume on contrary that $v_{i}v_{k},v_{k}v_{j}\in E$. Then by lemma \ref{lme3}, $v_{j}v_{k}\in E$ but this contradicts the fact that $G$ is oriented. Similarly one can prove for the case of $B$ when $v_{k}$ is not an inneighbour of $v_{f}$.
\end{proof}

\begin{lem}\label{lme6}
There is an ordering such that vertices of $A\cup N_{O}(A\setminus\{v_{f}\}) 
(B\cup N_{I}(B\setminus \{v_{f}\})$ satisfies (\ref{poicd}). Moreover there exists a subset of the above set which forms an Out-Fountain (In-Fountain) under the induced order. 
\end{lem}

\begin{proof}
We arrange the vertices of $N_{O}(A\setminus \{v_{f}\})$ in decreasing order of their out-degrees restricted on the set. We place the elements of $N_{O}(A\setminus \{v_{f}\})$ to the right of $v_{l}$. Let $\prec$ be the combined order on the set $A\cup N_{O}(A\setminus\{v_{f}\})$. Let $v_{1},v_{2},v_{3}$ be any three vertices.
If all of the three lie in $A$ then from Definition \ref{outin} they satisfy (\ref{poicd}). Again if all of them come from $N_{O}(A\setminus\{v_{f}\})$, using Lemma \ref{lme4} one can verify that they satisfy (\ref{poicd}). In view of Lemma \ref{lme3} this is also done if two of them lies in $A$. Now if $v_{1}\in A$, $v_{2},v_{3}\in N_{O}(A\setminus \{v_{f}\})$ and $v_{2}\prec v_{3}$, then let $v_{1}v_{3}\in E$. Clearly $v_{2}v_{3}\in E$ (from definition of $\prec$). It is sufficient to prove now $v_{1}v_{2}\in E$. Let $v_{j}$ be the first inneighbour of $v_{2}$ in $A$. Now if $v_{1}\prec v_{j}$ in $A$, then $v_{j}v_{3}\in E$ from Lemma \ref{lme3} as $v_{1}v_{3}\in E$. Moreover as $v_{j}v_{2}\in E$, $v_{2}v_{1}\notin E$ from Lemma \ref{lme5}. Therefore $\{v_{1},v_{j},v_{2},v_{3}\}$ induce $D_{3}$ in $G$ which is a contradiction. Hence $v_{j}\prec v_{1}$ and therefore from Lemma \ref{lme3} we get $v_{1}v_{2}\in E$.

\noindent Again in view of Lemma \ref{lme3} there exists a vertex $v_{i}\in A$ such that each element in $N_{O}(A\setminus \{v_{f}\})$ is an outneighbour of $v_{i}$ and hence for any $i\leq j\leq l.$ Therefore in the given ordering $\prec$ the set $\{v_{i},v_{i+1},\cdots , v_{l}\}\cup N_{O}(A\setminus \{v_{f}\})$ forms an Out-Fountain due to Lemma \ref{lme4}.  

\noindent Similarly one can prove $B\cup N_{I}(B\setminus \{v_{f}\})$ satisfies (\ref{poicd}) if we arrange the vertices of $N_{I}(B\setminus \{v_{f}\})$ in decreasing order of their in-degrees restricted on the set. Moreover there exists a set $B^{'}\subsetneq B$ for which $B^{'}\cup N_{I}(B\setminus \{v_{f}\})$ form an In-Fountain. 
\end{proof}

\begin{lem}\label{lme7}
Let $v_{k}\notin A$ and $v_{k}v_{i}\in E$ for some $f\leq i<l.$ If $v_{k}v_{l}\notin E$ then $v_{k}v_{j}\in E$ for all $f\leq j\leq i$ and $v_{k}v_{j}\in E$ for all $f\leq j\leq l$ if $v_{k}v_{l}\in E$ but $v_{f}v_{k}\notin E$. \\
\vspace{0.2em}
\noindent (Let $v_{k}\notin B$ and $v_{i}v_{k}\in E$ for some $f\leq i<l$. If $v_{l}v_{k}\notin E$ then $v_{j}v_{k}\in E$ for all $f\leq j \leq i$ and $v_{j}v_{k}\in E$ for all $f\leq j \leq l$ if $v_{l}v_{k}\in E$ but $v_{k}v_{f}\notin E$)
\end{lem}

\begin{proof}
As $G$ is acyclic, $v_{l}v_{k}\notin E$. Again if $v_{k}v_{l}\notin E$ then 
we reverse the vertex ordering of $A$, then by definition $A$ becomes an In-Fountain. Therefore by Lemma \ref{lme3} the result follows.
Now if $v_{k}v_{l}\in E$, then as $v_{k}v_{i}\in E$ and $v_{f}v_{k}\notin E$ using Lemma \ref{lem2} for the vertices $\{v_{f},v_{l},v_{i},v_{k}\}$
and Corollary \ref{corl1} for $N_{I}[v_{l}]$, we can conclude that $v_{k}v_{f}\in E$. Again as $v_{k}v_{l}\in E$, from Lemma \ref{lme5} we get $v_{k}v_{j}\in E$ for all $f\leq j\leq l$ as $G$ is chordal.  The proof for the case of $B$ is similar.
\end{proof}

\begin{defn}
Let $A_{0}=A$ be an Out-Fountain. We take $A_{1}= A\cup N_{O}(A\setminus \{v_{f}\})$. For $i>1$ let $A_{i+1}= A_{i}\cup N_{O}(A_{i}\setminus \{v_{f}\})$.
Then clearly, $A_{i}$ forms an increasing chain of subsets of $V$. As $V$ is a finite set, this chain must be stable. So there is an integer $N$ such that $A_{N}=A_{N+1}$. We call the set $A_{N}$ as {\em derived set of out-neighbours} \index{derived set} of $A$. Notationally we define it by $Der(A)$. Analogously we can define {\em derived set of in-neighbours} ($Der(B)$) for an In-Fountain $B$.
\end{defn}

\vspace{1em}

\noindent \textbf{Properties of Derived set:}
\vspace{0.6em}

\begin{enumerate}
\item No vertex of $Der(A)$ except $v_{f}$ has an out-neighbour outside $Der(A)$ and
no vertex of $Der(B)$ except $v_{f}$ has an in-neighbour outside $Der(B)$.

\item There exists an ordering on $Der(A)$ ($Der(B)$) satisfying (\ref{poicd}) and it induces the natural order in $A$.\label{p2}

\begin{proof}
Follows from repeated application of Lemma \ref{lme6}.
\end{proof}

\noindent  As $Der(A)$ is a ordered set from above property, by {\em boundary points} of $Der(A)$ we mean the two extreme vertices of the set in the induced order. By {\em interior point} of $Der(A)$ we mean any vertex of the set except the boundary points.
For $Der(B)$ these definitions are analogous. 

\item If $v_{k}\notin Der(A)$ is an in-neighbour of some interior point of $Der(A)$. Then it must also be an in-neighbour of $v_{f}$. (If $v_{k}\notin Der(B)$ is an out-neighbour of some interior point of $Der(B)$. Then it must also be an out-neighbour of $v_{f}$.)

\begin{proof}
We assume on contrary that $v_{k}v_{i}\in E$ but $v_{k}v_{f}\notin E$ where 
 $v_{i}$ is the first outneighbour of $v_{k}$ in $Der(A)$. Now if $v_{i}\in A\setminus \{v_{l}\}$ then from property $(1)$ it is clear that $v_{l}v_{k}\notin E$. Hence applying Lemma \ref{lme7} on vertices of $A$ it follows that $v_{k}v_{f}\in E$, which contradicts our assumption.
 
Again if $v_{i}\notin A\setminus \{v_{l}\}$ then let $v_{i}\in A_{t+1}\setminus A_{t}$ for some $t$ implies $v_{i}\in N_{O}[A_{t}]$ (say $N_{O}[v_{j}]$), then $v_{j}\prec v_{i}$ by our ordering. Also $v_{k}$ is not an outneighbour of 
$v_{j}$ from property $(1)$. Hence if $v_{i}$ is an interior point of $N_{O}[v_{j}]$ then from Lemma \ref{lme7} we get $v_{k}v_{j}\in E$, which is not true from the choice of $i$. Next if $v_{i}$ be the last end point of $N_{O}[v_{j}]$. Then we consider the sets $N_{O}[v_{j}]$ and $N_{O}[v_{i}]$. Again one can show that $v_{j},v_{k}$ are not adjacent in $U(G)$. Hence $D_{4}$ or $D_{2}$ get induced by the vertices $\{v_{j},v_{i},v_{k},v_{i_{2}}\}$ where $v_{i_{2}}$ is the last vertex of $N_{O}[v_{i}]$, which is a contradiction. Thus the result holds.
Similar proof will follow for the case of $Der(B)$. 
\end{proof}

\end{enumerate}

\begin{lem}\label{lme8}
Let $v\in V$ such that either $v$ has no in(out)-neighbour and $N_{O}[v]$
($N_{I}[v]$) is a single out(in)-neighbour tournament. Then there exists an ordering (say $\prec$) on $V$ satisfying (\ref{poicd}) having $v$ as the first element in $\prec$.
\end{lem}

\begin{proof}
We proceed by induction on $|V|$. Clearly done for $|V|=1$. If $|V|>1$ then consider $A=N_{O}[v]$. Let $V^{'}=(V\setminus Der(A))\cup \{v_{l^{'}}\}$ where $v_{l^{'}}$ is the last element of $Der(A)$. Then $|V^{'}|<|V|$. 
Also $v_{l^{'}}$ satisfies no out neighbour from property $(1)$ and it can induce maximum one inneighbour tournament (say $R_{v_{l^{'}}}$) whose vertices of are disjoint from $Der(A)$ (follows from Corollary \ref{corl1}, property $1,3$ and the fact $v$ has no inneighbour). Hence by induction there exists an ordering of $V^{'}$ satisfying (\ref{poicd}) and let that be $v_{l^{'}}\prec v_{l^{'}+1}\prec \hdots \prec v_{n}$. Also $Der(A)$ is an ordered set satisfying (\ref{poicd}) follows from property $2$ and hence $v=v_{1}\prec v_{2}\prec \hdots\prec v_{l^{'}}$. We now order the elements of $V$ as $v_{1}\prec v_{2}\prec \hdots v_{l^{'}}\prec v_{l^{'}+1}\prec \hdots \prec v_{n}$. This clearly satisfies (\ref{poicd}).
For the case of inneighbours we consider $Der(B)$ where $B=N_{I}[v]$ and proceed similarly.
\end{proof}

\noindent \textbf{Construction of $\hat{P}$}

\noindent There are three cases which may occur while constructing $\hat{P}$. Below we list all of them. As $v_{1}$ is the first element of $P$ one can easily check that $N_{I}[v_{1}]$ must be a single tournament. Further if $N_{O}[v_{1}]=A\cup B$ where $A,B$ are disjoint outneighbour tournaments of $v_{1}$, then $v_{1}$ can not have any inneighbour in $G$. On contrary if $v_{k}v_{1}\in E$ then from Lemma \ref{lem2} it follows that $v_{k}$ can not be simultaneously adjacent to $A$ and $B$ in $U(G)$. In other cases $D_{4}$ or $D_{5}$ get induced in $G$, which again introduce contradiction.

\begin{enumerate}

\item \textit{$N_{I}[v_{1}]$ is a single tournament and $v_{1}$ has no outneighbour.}

\begin{proof}
Construction of $\hat{P}$ follows from Lemma \ref{lme8}.
\end{proof}

\item \textit{$N_{O}[v_{1}]$ is disjoint union of two tournaments and $v_{1}$ has no inneighbour.}

\begin{proof}
\noindent Let $N_{O}[v_{1}]=A\cup B$. If we arrange the vertices of each sets individually according to decreasing outdegree restricted on the sets then both of them become Out-Fountains. Next we consider consider $Der(A)$. If $|Der(A)|>|A|$ then $Der(B)=B$ follows from the fact that all outneighbours of $v_{1}$ occur consecutively in $P$ right to $v_{1}$. As $A\cap B=\emptyset$ and they are tournaments, one can easily verify that each vertex of $B$ must occur prior every vertex of $A$ in $P$. Therefore if any vertex (say $v_{b}$) of $B$ possess inneighbour (say $v_{k}$) outside $V\setminus (A\cup B)$, then it must be an inneighbour of $A$ (say $v_{a}$). Therefore $\{v_{1},v_{a},v_{b},v_{k}\}$ form a $4$-cycle in $U(G)$. But as $v_{a},v_{b}$ are nonadjacent and $G$ is chordal, $v_{k}v_{1}\in E$. But this contradicts our assumption. Therefore using property $1$ we can conclude now that no vertex of $B\setminus \{v_{1}\}$ have any neighbour outside the set $B$.

We place $Der(A)$ right to $v_{1}$ in $\hat{P}$ and place the vertices of $B$ in reverse order left to $v_{1}$ so that it become an In-Fountain. Now applying Lemma \ref{lme8} one can order the vertices of $(V\setminus Der(B))\cup \{v_{1}\}$ so that they satisfy (\ref{poicd}). Again as $Der(B)$ satisfies (\ref{poicd}) from property $2$ it is easy to check now that $\hat{P}$ satisfies (\ref{poicd}).
Now if $Der(A)=A$ and $Der(B)=B$ then the proof is trivial.    
\end{proof}

\item \textit{$N_{O}[v_{1}]$ is a single tournament.}

\noindent If $v_{1}$ does not contain any inneighbour then $\hat{P}$ is formed by Lemma \ref{lme8}.

\noindent Again if $v_{1}$ has an inneighbour (say $v_{k}$) then $v_{l}$ can not have any outneighbour where $v_{l}$ is the last element of the Out-Fountain $N_{O}[v_{1}]$. More specifically $Der(A)=N_{O}[v_{1}]=A$.

Note that $v_{k}$ can not be an outneighbour of $v_{l}$ as $G$ is acyclic. On contrary let $v_{i}$ be an outneighbour of $v_{l}$. Now in $P$ as vertices of $N_{O}[v_{1}]$ occur consecutively right to $v_{1}$, $v_{i},v_{k}>v_{l}$ in $P$. Hence $v_{k}v_{l}\in E$. Now if $v_{k}<v_{i}$ in $P$ then $v_{l}v_{k}\in E$ which contradicts that $G$ is oriented. Again if $v_{k}>v_{i}$ then $\{v_{1},v_{l},v_{k},v_{i}\}$ induce $D_{9}$ which is a contradiction.

\noindent Now we will show $N_{I}[v_{l}]$ is a single tournament when $v_{k}v_{1}\in E$.

\noindent Again if on contrary we assume $v_{l}$ possesses an inneighbour $v_{i}$ which is disjoint from $v_{1}$, then $v_{i}\notin N_{O}[v_{1}]$ clearly. Now if $v_{i}<v_{k}$ in $P$, then the vertices $\{v_{1},v_{l},v_{k},v_{i}\}$ induce $D_{7}$, again if $v_{k}<v_{i}$, then they form $D_{3}$, which are both contradiction. Hence the result follows.

We can now apply Lemma \ref{lme8} on the In-Fountain $B=N_{I}[v_{l}]$(say). Thus we get an ordering $\hat{P}$ on $V$ satisfying (\ref{poicd}) having $v_{l}$ as the first element.
\end{enumerate}
\end{proof}

\begin{figure}\centering 
\begin{tikzpicture}
\draw (0,0)--(2,0);
\draw (0,0)--(0,2);
\draw (0,2)--(2,2);
\draw (2,0)-- (2,2);
\draw (0,0)--(2,2);

\draw  [fill=black](0,0) circle [radius=0.05];
\draw  [fill=black](2,0) circle [radius=0.05];
\draw  [fill=black](0,2) circle [radius=0.05];
\draw  [fill=black](2,2) circle [radius=0.05];

\node [left] at (0,0) {\tiny{$v_{1}$}};
\node [left] at (0,2) {\tiny{$v_{2}$}};
\node [right] at (2,2) {\tiny{$v_{3}$}};
\node [right] at (2,0) {\tiny{$v_{4}$}};

\draw[->] (0,2)--(0,1);
\draw [->] (2,0)--(1,0);
\draw[->] (2,2)--(1,2);
\draw[->] (2,0)--(2,1);
\draw[->] (2,2)--(1,1);

\draw (3,2)--(5,2);
\draw (5,2)--(7,2);
\draw (5,2)--(5,0);
\draw (7,2)--(5,0);

\draw  [fill=black](3,2) circle [radius=0.05];
\draw  [fill=black](5,2) circle [radius=0.05];
\draw  [fill=black](7,2) circle [radius=0.05];
\draw  [fill=black](5,0) circle [radius=0.05];

\node [above] at (3,2) {\tiny{$v_{1}$}};
\node [above] at (5,2) {\tiny{$v_{2}$}};
\node [above] at (7,2) {\tiny{$v_{3}$}};
\node [left] at (5,0) {\tiny{$v_{4}$}};

\draw[->] (5,2) -- (6,2);
\draw[->] (5,0) -- (5,1);
\draw[->] (5,0) -- (6,1);

\draw (9,2)--(10,1);
\draw (10,1)--(11,2);
\draw (10,1)--(10,0);

\draw  [fill=black](9,2) circle [radius=0.05];
\draw  [fill=black](11,2) circle [radius=0.05];
\draw  [fill=black](10,1) circle [radius=0.05];
\draw  [fill=black](10,0) circle [radius=0.05];

\draw[->] (10,1)--(9.5,1.5);
\draw [->] (10,1)--(10.5,1.5);

\node [left] at (9,2) {\tiny{$v_{1}$}};
\node [left] at (10,1) {\tiny{$v_{2}$}};
\node [right] at (11,2) {\tiny{$v_{3}$}};
\node [left] at (10,0) {\tiny{$v_{4}$}};

\draw (0,-1)--(2,-1);
\draw (0,-3)--(2,-3);
\draw (0,-1)--(0,-3);
\draw (2,-1)-- (2,-3);
\draw (2,-1)--(0,-3);

\draw  [fill=black](0,-1) circle [radius=0.05];
\draw  [fill=black](0,-3) circle [radius=0.05];
\draw  [fill=black](2,-1) circle [radius=0.05];
\draw  [fill=black](2,-3) circle [radius=0.05];

\node [left] at (0,-1) {\tiny{$v_{2}$}};
\node [right] at (2,-1) {\tiny{$v_{3}$}};
\node [left] at (0,-3) {\tiny{$v_{1}$}};
\node [right] at (2,-3) {\tiny{$v_{4}$}};

\draw[->] (0,-3)--(0,-2);
\draw[->] (0,-3)--(1,-3);
\draw[->] (0,-3)--(1,-2);
\draw[->]  (2,-1)--(1,-1);
\draw[->] (2,-1)--(2,-2);

\draw (3.5,-1)--(5.5,-1);
\draw (3.5,-3)--(5.5,-3);
\draw (3.5,-1)--(3.5,-3);
\draw (5.5,-1)-- (5.5,-3);
\draw (5.5,-1)--(3.5,-3);

\draw  [fill=black](3.5,-1) circle [radius=0.05];
\draw  [fill=black](3.5,-3) circle [radius=0.05];
\draw  [fill=black](5.5,-1) circle [radius=0.05];
\draw  [fill=black](5.5,-3) circle [radius=0.05];

\node [left] at (3.5,-1) {\tiny{$v_{2}$}};
\node [right] at (5.5,-1) {\tiny{$v_{3}$}};
\node [left] at (3.5,-3) {\tiny{$v_{1}$}};
\node [right] at (5.5,-3) {\tiny{$v_{4}$}};

\draw[->] (5.5,-1)--(4.5,-1);
\draw[->] (5.5,-1)--(5.5,-2);
\draw[->] (3.5,-1)--(3.5,-2);
\draw[->]  (5.5,-3)--(4.5,-3);
\draw[->] (5.5,-1)--(4.5,-2);

\draw (8,-1)--(10,-1);
\draw (8,-3)--(10,-3);
\draw (8,-1)--(8,-3);
\draw (10,-1)--(10,-3);
\draw (8,-3)--(10,-1);

\draw  [fill=black](8,-1) circle [radius=0.05];
\draw  [fill=black](10,-1) circle [radius=0.05];
\draw  [fill=black](8,-3) circle [radius=0.05];
\draw  [fill=black](10,-3) circle [radius=0.05];

\node [left] at (8,-1) {\tiny{$v_{2}$}};
\node [right] at (10,-1) {\tiny{$v_{3}$}};
\node [left] at (8,-3) {\tiny{$v_{1}$}};
\node [right] at (10,-3) {\tiny{$v_{4}$}};

\draw[->] (8,-1)--(9,-1);
\draw[->] (8,-1)--(8,-2);
\draw[->] (8,-3)--(9,-2);
\draw[->] (10,-3)--(10,-2);
\draw[->] (10,-3)--(9,-3);

\draw (12,-1)--(14,-1);
\draw (12,-3)--(14,-3);
\draw (12,-1)--(12,-3);
\draw (14,-1)--(14,-3);
\draw (12,-3)--(14,-1);

\draw  [fill=black](12,-1) circle [radius=0.05];
\draw  [fill=black](14,-1) circle [radius=0.05];
\draw  [fill=black](12,-3) circle [radius=0.05];
\draw  [fill=black](14,-3) circle [radius=0.05];

\node [left] at (12,-1) {\tiny{$v_{2}$}};
\node [right] at (14,-1) {\tiny{$v_{3}$}};
\node [left] at (12,-3) {\tiny{$v_{1}$}};
\node [right] at (14,-3) {\tiny{$v_{4}$}};

\draw[->] (14,-1)--(13,-1);
\draw[->] (12,-1)--(12,-2);
\draw[->] (14,-1)--(14,-2);
\draw[->] (12,-3)--(13,-3);
\draw[->] (14,-1)--(13,-2);

\node [] at (1,-0.5) {\tiny{$D_{3}$ }};
\node [] at (5,-0.5) {\tiny{$D_{4}$}};
\node [] at (10,-0.5) {\tiny{$D_{5}$}};
\node[] at (1,-3.5){\tiny{$D_{6}$}};
\node[] at (4.5,-3.5){\tiny{$D_{7}$}};
\node[] at (9,-3.5){\tiny{$D_{8}$}};
\node[] at (13.5,-3.5){\tiny{$D_{9}$}};

\end{tikzpicture}
\caption{oriented ICD but not proper oriented ICD (undirected edge can have any direction)}\label{forbid}
\end{figure}

\begin{figure}\centering
\begin{tikzpicture}

\draw (0,0)--(2,0);
\draw(0,0)--(0,2);
\draw (2,0)--(2,2);
\draw (0,2)--(2,2);
\draw (0,0)--(2,2);

\draw  [fill=black](0,0) circle [radius=0.05];
\draw  [fill=black](2,0) circle [radius=0.05];
\draw  [fill=black](0,2) circle [radius=0.05];
\draw  [fill=black](2,2) circle [radius=0.05];

\draw[->] (0,2)--(1,2);
\draw[->] (0,2)--(0,1);
\draw[->] (0,0)--(1,0);
\draw[->] (2,2)--(2,1);

\draw[->] (2,2)--(1,1);

\end{tikzpicture}
\caption{$\hat{D}$}\label{allowed} 
\end{figure}

\section{Conclusion}
It was proved in \cite{Soto} that interval graphs $\subset$ central MPTG $\subseteq$ max-tolerance graph. Combining with these we establish the relations between some subclasses of max-tolerance graphs related to central MPTG in Figure \ref{subclass}. 

\noindent Note that $C_{4}\in \text{unit max-tolerance}\smallsetminus \text{interval graphs}$. $C_{4}$ has a unit-max-tolerance representation having intervals $[1,5],[2,6],[3,7],[4,8]$ and corresponding tolerances $1,3,3,1$ for its consecutive vertices (clockwise or anticlockwise). But $C_{4}$ is not an interval graph. In Theorem \ref{c61} we have shown that $\overline{C_{6}}$ is a max-tolerance graph which is not a central MPTG. Also we note that $K_{2,3}$ is not a central MPTG (by Lemma $7$ of \cite{Soto}) but it is a MPTG (by Lemma $8$ of \cite{Soto}). Again  $K_{1,3}\in 50\%\ \text{max-tolerance graph}\smallsetminus \text{proper interval graph}$. $K_{1,3}$ is an example of a graph which has a $50\%$ max-tolerance representation having interval $[1.9,6.1]$ and tolerance $2.1$ for its central vertex and the intervals $[0,8],[1.8,4.3],[3.6,5.9]$ and corresponding tolerances $4,1.25,1.15$ for pendant vertices. But it is not a proper interval graph.

\vspace{.2em}\noindent The following examples lead us to conclude that interval graphs and $50\%$ max-tolerance graphs are incomparable.

\vspace{.2em}\noindent $C_{4}\in 50\%\ \text{max-tolerance graph}\smallsetminus \text{interval graph}$. $C_{4}$ is an example of a graph which has $50\%$ max-tolerance representation having intervals $[1,4.6],[2,4],[2.9,4.9],[2.7,6.3]$ and corresponding tolerances $1.8,1,1,1.8$ for its consecutive vertices (clockwise or anticlockwise). But it is not an interval graph. Again $K_{1,n}$, $n\geq 8 \in \text{interval Graphs}\smallsetminus 50\%\ \text{max-tolerance graph}$. $K_{1,n}$ is an interval graph having interval $[1,2n]$ for central vertex and intervals $[2i-1,2i]$ for pendant vertices where $i\in\{1,\hdots,n\}$. But it has no $50\%$ max-tolerance representation follows from Lemma \ref{px4}.

\noindent Again from Theorem \ref{com} we get to know that central MPTG and $50\%$ max-tolerance graphs are not comparable. One can also conclude from above that $K_{1,n},n\geq 8\in \text{max-tolerance graph}
$. But it is not a $50\%$ max-tolerance graph from Theorem \ref{px4}. 

\noindent In this paper we consider three subclasses of the class of ICD, namely central ICD, oriented ICD and proper ICD. We obtain characterizations for central ICD, oriented ICD when it is a tournament and proper ICD. Also in terms of forbidden subdigraphs we characterize those oriented ICD for which their underlying graph is tree and those proper oriented ICD for which their underlying graph is chordal. In Figure \ref{relation1} we show relationships between the digraphs discussed in this paper whose proofs follow from the mentioned characterization theorems. Several combinatorial optimization problems remain open for these digraphs, for example, to find the complete list of forbidden digraphs or to construct recognition algorithms for them. 

\vspace{.2em}\noindent
Finally we note the major unsolved problems in this area.
\begin{enumerate}
\item Recognition algorithm and forbidden subgraph characterization of max-point tolerance graphs.
\item Combinatorial characterization, adjacency matrix characterization, recognition algorithm and forbidden subgraph characterization of central max-point tolerance graphs and $50\%$ max-tolerance graphs.
\item Find the complete list of forbidden digraphs of central ICD, proper ICD.
\end{enumerate}

\begin{figure}[t]

\begin{center}
\begin{tikzpicture}
\node at (2.5,1) {\Fontvix{max-tolerance graph}};
\node at (-0.2,1) {\Fontvix{MPTG}};

\node at (0.6,-1) {\Fontvix{central MPTG=unit max-tolerance graph}};
\node [below] at (1,-2.5) {\Fontvix {interval graph}};
\node[below] at  (1,-4.5) {\Fontvix {proper interval graph=proper central MPTG=unit central MPTG}};
\node[below] at  (.54,-3.5) {\tiny {$K_{1,3}$}};
\node[below] at  (.64,-1.7) {\tiny {$C_{4}$}};
\node[below] at  (1.64,0.2) {\tiny {$\overline{C_{6}}$}};
\node[below] at  (0,0.2) {\tiny {$K_{2,3}$}};
\draw (1.7,.8)--(1,-.8);
\draw (0,0.8)--(1,-0.8);
\draw (1,-1.2)--(1,-2.5);
\draw (1,-2.9)--(1,-4.5);

\draw (1,-4.5)--(3.7,-1.1);
\draw (3.7,-0.9)--(2.5,0.9);
\node at (5.1,-1) {\Fontvix{$50\%$ max-tolerance graph}};
\node[below] at (3,-2.5) {\tiny{ $K_{1,3}$}};
\node[below] at (3.64 ,0.2 ) {\tiny{$K_{1,8}$}};

\end{tikzpicture}
\caption{Hierarchy of subclasses of the class of max-tolerance graph}\label{subclass}
\end{center}

\end{figure}

\begin{center}
\begin{figure}[b]

	\begin{tikzpicture}[scale=1]
	
	  \draw (1,0.5)--(7,0.5);
    \draw (7,0.5)--(7,5.3);
    \draw (1,5.3)--(7,5.3);
    \draw (1,5.3)--(1,0.5);
    
    \node at (3.7,5){{\footnotesize ICD}};
    \node at (2.4,4.8){$D_{11}$};
    
    \draw (1.5,1)--(5,1);
    \draw (5,1)--(5,4.3);
    \draw (1.5,4.3)--(5,4.3);
    \draw (1.5,1)--(1.5,4.3);
    
    \node at (2,3.5){$D_{10}$};
    \node at (2.6,4.1){{\footnotesize Central}};
    \node at (2.7,3.7){{\footnotesize ICD}}; 
    
    \draw (1.8,1.2)--(4.7,1.2);
    \draw (4.7,1.2)--(4.7,3.2);
    \draw (4.7,3.2)--(1.8,3.2);
    \draw (1.8,3.2)--(1.8,1.2);   
    		
 	\node [left] at (2.7,1.7) {$D_{12}$}; 	
 	
 	\node[below] at (2.6,3){{\footnotesize Oriented}};
 	\node[below] at (2.5,2.6){{\footnotesize ICD}};
 	
 	\node[below] at (4,3){{\footnotesize Proper}};
 	\node[below] at (4,2.7){{\footnotesize oriented}};
 	\node[below] at (4,2.3){{\footnotesize ICD}};
		
	\draw (3.3,1.5)--(6.7,1.5);
	\draw (6.7,1.5)--(6.7,4.6);	
	\draw (6.7,4.6)--(3.3,4.6);
	\draw (3.3,1.5)--(3.3,4.6);
		
	\node[below] at (6,3.2){$G_1$};
\node at (6,4.2){{\footnotesize Proper}};
\node at (6,3.8) {{\footnotesize ICD}};
\draw (8,5)--(10,5);
\draw (8,5)--(9,3);
\draw (10,5)--(9,3);
\draw (9,3)--(9,2);
\draw  [fill=black](8,5) circle [radius=0.05];
\draw  [fill=black](10,5) circle [radius=0.05];
\draw  [fill=black](9,3) circle [radius=0.05];
\draw  [fill=black](9,2) circle [radius=0.05];
\node [left] at(8,5) {$v_{1}$};
\node [right] at(10,5) {$v_{2}$};
\node [below] at(9.3,3) {$v_{3}$};
\node [below] at(9,2) {$v_{4}$};
\draw[->](8,5)--(9,5);
\draw[->](9,3)--(8.5,4);
\draw[->](10,5)--(9.5,4);
\draw[->](9,3)--(9.5,4);
\draw[->](9,3)--(9,2.5);
\draw[->](9,2)--(9,2.5);
\node at(9,1) {$D_{10}$};
\draw (12,5)--(14,5);
\draw (14,5)--(16,5);
\draw (14,5)--(14,3.5);
\draw (12,5)--(14,3.5);
\draw (16,5)--(14,3.5);
\draw (14,3.5)--(14,2);
\draw (12,5)--(14,2);
\draw  [fill=black](12,5) circle [radius=0.05];
\draw  [fill=black](14,5) circle [radius=0.05];
\draw  [fill=black](16,5) circle [radius=0.05];
\draw  [fill=black](14,3.5) circle [radius=0.05];
\draw  [fill=black](14,2) circle [radius=0.05];
\node [left] at(12,5) {$v_{4}$};
\node [right] at(16,5) {$v_{2}$};
\node [above] at(14,5) {$v_{1}$};
\node [right] at(14,3.2) {$v_{3}$};
\node [below] at(14,2) {$v_{5}$};
\draw[->](12,5)--(13,5);
\draw[->](14,5)--(13,5);
\draw[->](14,5)--(15,5);
\draw[->](14,5)--(14,4);
\draw[->](16,5)--(15,4.25);
\draw[->](12,5)--(13,4.25);
\draw[->](14,3.5)--(13,4.25);
\draw[->](14,3.5)--(14,3);
\draw[->](14,2)--(13,3.5);
\draw (12,5) to[out=90,in=90] (16,5);
\draw[->](13.9,6.17)--(14,6.17);
\node at(14,1) {$D_{11}$};
\draw (18,5)--(18,3.5);
\draw (16,3.5)--(18,3.5);
\draw (18,3.5)--(18,2);
\draw  [fill=black](18,5) circle [radius=0.05];
\draw  [fill=black](18,3.5) circle [radius=0.05];
\draw  [fill=black](18,2) circle [radius=0.05];
\draw  [fill=black](16,3.5) circle [radius=0.05];
\node [above] at(18,5) {$v_{1}$};
\node [right] at(18,3.5) {$v_{2}$};
\node [below] at(16,3.5) {$v_{4}$};
\node [left] at(18,2) {$v_{3}$};
\draw[->](18,3.5)--(18,4.25);
\draw[->](18,3.5)--(18,2.75);
\draw[->](18,3.5)--(17,3.5);
\node at (18,1) { $D_{12}$};
\end{tikzpicture}
\caption{Relations between some subclasses of ICD}
\label{relation1}
\end{figure}
\end{center}

\noindent\textbf{Acknowledgements:}

\noindent This research is supported by UGC (University Grants Commission) NET fellowship (21/12/2014(ii)EU-V) of the first author.

\end{document}